\documentclass[11pt, reqno, a4paper]{amsart}
\usepackage[margin=1.2in]{geometry}
\numberwithin{equation}{section}
\usepackage{amssymb,amsfonts,amsthm}
\usepackage[utf8]{inputenc}
\usepackage{listings,cmap}
\usepackage{bm}
\usepackage[hyperpageref]{backref}
 \usepackage{tikz}
\usetikzlibrary{arrows.meta,arrows}
\usepackage[shortlabels]{enumitem}
\usepackage{epsfig}

\usepackage{mathtools}
\usepackage[hyperpageref]{backref}

\mathtoolsset{showonlyrefs}

\usepackage{color}
\usepackage[bookmarks]{hyperref}

\allowdisplaybreaks[4]

\newtheoremstyle{myremark}{10pt}{10pt}{}{}{\bfseries}{.}{.5em}{}

 \newtheorem{thm}{Theorem}[section]
 
 \newtheorem{lem}[thm]{Lemma}
 \newtheorem{prop}[thm]{Proposition}
 \theoremstyle{definition}
 \newtheorem{defn}[thm]{Definition}
 \newtheorem{exmp}[thm]{Example}
 \newtheorem{rem}[thm]{Remark}

 \newcommand{\D}{\bigtriangleup_2}

\allowdisplaybreaks[4]

\title[On fractional Orlicz Boundary Hardy inequalities]{On fractional Orlicz  boundary Hardy inequalities}

 \author[S. Roy]{ Subhajit Roy}



 {\email{rsubhajit.math@gmail.com}}

\begin{document}

\maketitle

 \centerline{Department of Mathematics, Indian Institute of Technology Madras,
 }
 \centerline{Chennai  600036, India}

\begin{abstract} 
We investigate the fractional Orlicz boundary Hardy-type inequality for bounded Lipschitz domains. Further, we establish fractional Orlicz boundary Hardy-type inequalities with logarithmic corrections for specific critical cases across various domains, such as bounded Lipschitz domains, domains above the graph of a Lipschitz function, and the complement of a bounded Lipschitz domain.
\end{abstract}

\section{Introduction}
For  $p\in(1,\infty)$, recall the {\it{boundary Hardy inequality}} for a bounded Lipschitz domain $\Omega\subset\mathbb{R}^N$:  
\begin{equation}\label{bd-hardy}
    \int_{\Omega}\frac{|u(x)|^p}{\delta^p_\Omega(x)}dx\leq C\int_{\Omega}|\nabla u(x)|^pdx,\quad\forall\,u\in\mathcal{C}_c^1(\Omega),
\end{equation}  
for some $C=C(N,p,\Omega)>0,$ where  $\delta_\Omega(x):=\min_{y\in \partial\Omega} |x-y|$ (see \cite{1962}).  
 This inequality has been further
developed for various domains such as: domains with H\"older boundary \cite{Kufner1985}, unbounded John domains \cite{Juha2008}, and domains with uniformly p-fat complement \cite{Lewis1998}. For more details on the boundary Hardy inequalities, see \cite{brezis20000,ujjal2023,Pinchover,Matskewich1997} and the references therein. For the generalization of \eqref{bd-hardy} in the Orlicz setting, we refer to  \cite{subha,subha2,Mihai2012,kal2009} for $\Omega=\mathbb{R}^N\setminus\{0\}$, \cite{cianchi1999} for bounded Lipschitz domain, and \cite{Buckley2013,Buckley2004} for  bounded domains with uniformly p-fat complement. We are interested in studying the fractional version of these inequalities, given its significant applications in stochastic processes (see, for example, \cite{Bogdan22003, dyda2011, chen2003, Dyda2007} and the references therein). We fix a few notations to state the fractional version of \eqref{bd-hardy}.
 For $u\in \mathcal{C}_c(\Omega)$ and $s\in (0,1),$ let $D_su$ be the s-H\"older quotient and $d\mu$ be the product measure on $\Omega\times \Omega$ defined as $$D_su(x,y)=\frac{u(x)-u(y)}{|x-y|^s},\quad d\mu=\frac{dxdy}{|x-y|^N}.$$ We consider the following classes of domains in  $\mathbb{R}^N:$
\begin{enumerate}
    \item[$\mathcal{A}_1$] $=\left\{\Omega:\Omega\; \text{is a bounded Lipschitz domain} \right\}$,
     \item[$\mathcal{A}_2$] $=\left\{\Omega: \Omega\; \text{is a domain above the graph of a Lipschitz function}\; \mathbb{R}^{N-1}\rightarrow\mathbb{R}\right\}$,
     
      \item[$\mathcal{A}_3$]  $=\left\{\Omega:\Omega\; \text{is a complement of a bounded Lipschitz domain}\right\}$,
      \item[$\mathcal{A}_4$] $=\left\{\Omega:\Omega\; \text{is a complement of a point}\right\}$.
\end{enumerate}
For $p\in (1,\infty),s\in(0,1)$, and $\Omega$ satisfying one of the following assumptions: 
\begin{enumerate}[(i)]
    \item $\Omega\in \mathcal{A}_1$ and $sp>1$,
    \item $\Omega\in \mathcal{A}_2$ and $sp\neq 1$,

    \item $\Omega\in \mathcal{A}_3$, $sp\ne 1,$ and $sp\ne N$,
     \item $\Omega\in \mathcal{A}_4$ and $sp\neq N$,
\end{enumerate}
 Dyda \cite[Theorem 1.1]{Dyda2007} established the following fractional boundary Hardy inequality
 \begin{equation}\label{frac-dyda}
   \int_{\Omega}\left(\frac{|u(x)|}{\delta^s_\Omega(x)}\right)^p dx\leq C \int_{\Omega}\int_{\Omega}|D_su(x,y)|^p d\mu,\quad \,\forall\,u\in\mathcal{C}_c^1(\Omega),
\end{equation}
where $C$ is a positive constant independent of $u$. The above inequality is also established for convex domains \cite{Brasco2018,Loss2010}, unbounded John domains \cite{john}, and domains with uniformly p-fat complement \cite{Edmunds2014}. The best constant of \eqref{frac-dyda} for the upper half-space is addressed for $p = 2$ in \cite{dyda2011} and for $p \geq 1$ in \cite{frank2010}. 

\smallskip

The above inequality fails for the critical cases: $sp=1$ for $\Omega\in\mathcal{A}_1$ \cite[Section 2]{Dyda2007} and $\Omega\in\mathcal{A}_2$ \cite[Appendix]{adisahuroy}, $sp=N$ for $\Omega$ in $\mathcal{A}_3$ and $ \mathcal{A}_4$  \cite[Section 2]{Dyda2007}. The critical case $sp=N$ for $\Omega=\mathbb{R}^N\setminus\{0\}$ is studied by Edmunds and Triebel \cite{triebel1999}. They  proved an appropriate version of \eqref{frac-dyda} by adding a logarithmic
weight on the left-hand side. Nguyen and Squassina \cite{Nguyen2018} extended their result and established the full range of the fractional Caffarelli-Kohn-Nirenberg inequality. Adimurthi, Jana, and Roy \cite{adipurbroy2} studied the critical case $sp=1$ for $\Omega\in \mathcal{A}_1$ in one dimension. For $\Omega$ in $\mathcal{A}_1$, $\mathcal{A}_2$, and $\mathcal{A}_3,$ the critical cases are studied by Adimurthi, Roy, and Sahu \cite{adisahuroy}.   In fact, their results are much more general. More precisely, they proved that the following Hardy-type inequality 
   (see \cite[Theorem 2]{adisahuroy} with $\tau=p$):
\begin{equation}\label{critical_cases}
 \int_{\Omega}\left(\frac{|u(x)|}{\delta_\Omega^{s}(x)}\right)^p\frac{dx}{\ln^{p}\left(\rho(x)\right)}\leq C\left(\int_{\Omega}\int_{\Omega}|D_su(x,y)|^pd\mu+\int_\Omega|u(x)|^pdx\right),\,\, \forall\,u\in \mathcal{C}_c^1(\Omega)
\end{equation}
holds in each of the following cases:
\begin{enumerate}[(i)]
    \item $\Omega\in \mathcal{A}_1$, $sp=1, $ and $\rho(x)=\frac{2R}{\delta_\Omega(x)},$ where $R\ge R_\Omega:=\sup\{\delta_\Omega(x):x\in\Omega\}$
    \item $\Omega\in \mathcal{A}_2$, $sp=1, $ $supp(u)\subset\{x\in \Omega: \delta_\Omega(x)<R\}$,   and $\rho(x)=\frac{2R}{\delta_\Omega(x)}$ for some $R>0$
    \item $\Omega\in \mathcal{A}_3$, $sp=N,\;N>1, $    and $\rho(x)=\max\left\{\frac{2R}{\delta_\Omega(x)},\frac{2\delta_\Omega(x)}{R}\right\}$ for some $R>0$.
\end{enumerate}

\smallskip

It is known that \eqref{frac-dyda} fails for $\Omega\in \mathcal{A}_1$ in the case $sp< 1$, see, for example, \cite[Page 578]{Dyda2007}. For bounded $\mathcal{C}^\infty$ domains $\Omega$ and $sp<1,$ Triebel \cite[Page 259]{Hans1978} added the $L^p$-norm of $u$ to the right hand of \eqref{frac-dyda} and established
\begin{equation}\label{dyda-2}
   \int_{\Omega}\left(\frac{|u(x)|}{\delta^s_\Omega(x)}\right)^p dx\leq C \left(\int_{\Omega}\int_{\Omega}|D_su(x,y)|^p d\mu+\int_\Omega|u(x)|^pdx\right),\quad \,\forall\,u\in\mathcal{C}_c^\infty(\Omega)
\end{equation}
where $C=C(s,p,N,\Omega)>0.$ This result is extended to $\Omega\in \mathcal{A}_1$ and $u\in \mathcal{C}_c^1(\Omega)$ in \cite{chen2003} for $p=2$, and in \cite{adisahuroy, Dyda2007} for $p>1$.

\smallskip

\smallskip

The aim of this article is to   extend \eqref{critical_cases} for $\Omega\in\mathcal{A}_1, \mathcal{A}_2,\mathcal{A}_3$ and \eqref{dyda-2} for $\Omega\in\mathcal{A}_1,$ by replacing the convex function $t^p$ with a more general Young function that satisfies the $\D$-condition.
\begin{defn}[Young function]\label{defn-orlicz}
  A  function  $\Phi:[0,\infty)\to [0,\infty)$ is called a Young function if it can be represented in the form
 \begin{equation*}
     \Phi(t)=\int_0^t \varphi(s) ds\quad   \text{for}\,\, t\geq 0,
 \end{equation*} 
 where $\varphi$ is a non-decreasing right continuous function on $[0,\infty)$ satisfying $\varphi(0)=0,$ $\varphi(t)>0$ for $t>0,$ and $\lim_{t\to \infty}\varphi(t)=\infty$.
    % A  continuous, convex function  $\Phi:[0,\infty)\to [0,\infty)$ is called a Young function if it has the following properties:
    % \begin{enumerate}[(a)]
    %     \item $\lim_{t\to 0}\frac{\Phi(t)}{t}=0,$
    %     \item  $\lim_{t\to \infty}\frac{\Phi(t)}{t}=\infty$.
    %       \end{enumerate}
         \end{defn}
         \begin{defn}[The $\D$-condition]
     We say that the Young function $\Phi$ satisfies the  $\D$-condition or $ \Phi\in \D$ if there exists a constant $C> 0$ such that 
\begin{equation*}
    \Phi(2t)\leq C\Phi(t),\;\;\;\forall\,t\geq 0.
\end{equation*}
         \end{defn}
         To state our results, recall the indices $p^-_\Phi$ and $p^+_\Phi$ (cf. \cite{Mihai2012}):
 \begin{equation*}
p_\Phi^-:=\inf_{t>0}\frac{t\varphi(t)}{\Phi(t)},\qquad p_\Phi^+:=\sup_{t>0}\frac{t\varphi(t)}{\Phi(t)}.
 \end{equation*}
 %where $\varphi$ is the right continuous derivatives of $\Phi$ (Definition \ref{defn-orlicz}).
 For a Young function $\Phi\in \D$, it can be verified that 
 $ 1\le p_\Phi^- \leq p_\Phi^+< \infty.$ 
For example, for $\Phi(t)=t^p$ with $p\in (1,\infty)$, we have $p_\Phi^-=p_\Phi^+=p.$ For more examples, see \cite[Example 5.5]{subha2} (Example \ref{example-p3}) and the references therein.

\smallskip

In \cite{roy2022}, the authors extended \eqref{frac-dyda} by replacing $t^p$ with a Young function satisfying the $\D$-condition. More precisely, they have provided sufficient conditions on $s,N,\Phi$ so that the following fractional Orlicz boundary Hardy  inequality holds:
\begin{equation}\label{Hardy-p3}
    \int_{\Omega}\Phi\left(\frac{|u(x)|}{\delta^s_\Omega(x)}\right)dx\leq C\int_{\Omega}\int_{\Omega}\Phi\left(|D_su(x,y)|\right)d\mu, \quad \forall\,u\in \mathcal{C}_c^1(\Omega),
\end{equation} 
where $C$ is positive constant independent of $u.$ The above inequality is established for $\Omega$ in $\mathcal{A}_1,\,\mathcal{A}_2$, $\mathcal{A}_3$, and $\Omega=\mathbb{R}^N\setminus\{0\}\in \mathcal{A}_4$.  Inequality \eqref{Hardy-p3} is also proved in \cite{salort2022} for $\Omega=(0,\infty)$ and in \cite{Cianchi2022,subha2,  salort2022*} for $\Omega=\mathbb{R}^N\setminus\{0\}$. However, for $\Omega\in\mathcal{A}_1,$ the above inequality fails in the case $sp^+_\Phi\le 1$ (see \cite[Theorem 1.3]{roy2022}). %It is known that \eqref{dyda-2} also fails in the critical case $sp=1$ for $\Omega\in \mathcal{A}_1$; see, for example, \cite{Bogdan22003}.

\smallskip

In this article, first, we extend \eqref{critical_cases} for $\Omega\in \mathcal{A}_1$ in the critical case $sp^+_\Phi=1$ by replacing the convex function $t^p$ with a more general Young function satisfying the $\D$-condition. In particular, we have the following result.
\begin{thm}\label{Orlicz-Hardy2-p3}  Let $N\geq 1,\,s\in (0,1),$ and $\Omega\in \mathcal{A}_1$. Let $R_\Omega:=\sup\{\delta_\Omega(x):x\in\Omega\}$ and $R\ge R_\Omega$. 
For any Young function $\Phi\in \D,$  if $sp^+_\Phi=1,$ then there exists $C=C(s,N,\Phi,\Omega)>0$ so that 
  \begin{equation}\label{2}
 \int_{\Omega}\Phi\left(\frac{|u(x)|}{\delta_\Omega^{s}(x)}\right)\frac{dx}{\ln^{p_\Phi^+}\left(\frac{2R}{\delta_\Omega(x)}\right)}\leq C\left(\int_{\Omega}\int_{\Omega}\Phi\left(|D_su(x,y)|\right)d\mu+\int_\Omega\Phi(|u(x)|)dx\right),\,\, \forall\,u\in \mathcal{C}_c^1(\Omega).
\end{equation}
\end{thm}

\smallskip

Now, we consider the case $sp^+_\Phi<1$ for $\Omega\in \mathcal{A}_1.$ 
In this case, we have the following extension of \eqref{dyda-2}:
\begin{thm}\label{Orlicz-Hardy1*-p3} Let $N\geq 1,\,s\in (0,1),$ and $\Omega\in \mathcal{A}_1$. For any Young function $\Phi\in \D,$  if $sp^+_\Phi<1,$ then there exists $C=C(s,N,\Phi,\Omega)>0$ so that 
\begin{equation}\label{1}
 \int_{\Omega}\Phi\left(\frac{|u(x)|}{\delta^s_\Omega(x)}\right)dx\leq C\left(\int_{\Omega}\int_{\Omega}\Phi\left(|D_su(x,y)|\right)d\mu+\int_\Omega\Phi(|u(x)|)dx\right),\quad\forall\,u\in \mathcal{C}_c^1(\Omega).
\end{equation}
\end{thm}
\smallskip
% Our proof is based on the dyadic decomposition of RN , Poincar ́e inequalities for annulus, and
% a clever summation process to pass the information from a family of annuli to the whole space.
% The similar ideas are used in [37] and [38].

For proving Theorem \ref{Orlicz-Hardy2-p3} and Theorem \ref{Orlicz-Hardy1*-p3}, we adapt the idea of proof of Theorem 2 and Theorem 3 of \cite{adisahuroy} (see also \cite[Theorem 1.3]{adipurbroy2}) in the Orlicz setting.
 In their proof,  they have used a crucial lemma (see Lemma 3.2 of \cite{Nguyen2018}), which states that for $p,\Lambda\in(1,\infty),$ there exists $C=C(p,\Lambda)>0$ so that 
     \begin{equation}\label{in-1}
(a+b)^p\leq \lambda a^p+\frac{C}{(\lambda-1)^{p-1}}b^p,\;\;\;\forall\,a,b\in [0,\infty),\,\forall\,\lambda\in (1,\Lambda). 
  \end{equation} 
A less restricted version of the above inequality is also proved in \cite[Lemma 2.4]{adisahuroy}. This article used the Orlicz version of the above inequality established by Anoop, P. Roy, and S. Roy in \cite{subha2}. It says that 
 for any Young function $\Phi\in \D$ and $\Lambda>1,$ there exists $C=C(\Phi,\Lambda)>0$ so that
     \begin{equation}\label{in-2}
 \Phi(a+b)\leq \lambda\Phi(a)+\frac{C}{(\lambda-1)^{p^+_\Phi-1}}\Phi(b),\;\;\;\forall\,a,b\in [ 0,\infty),\,\forall\,\lambda\in (1,\Lambda].  
  \end{equation} 
% The proof \eqref{in-1} is based on the homogeneity of $t^p$. On the other hand, the proof of \eqref{in-2} is by using some subtle properties of the Orlicz function.

% Our proof of Theorem \ref{Orlicz-Hardy1*-p3} and Theorem \ref{Orlicz-Hardy2-p3} is based on a Young-type inequality and the Poincaré-Wirtinger inequality for an Orlicz function \cite{subha2} (see also Lemma \ref{Lem-H-p3} and Lemma \ref{poincare}). By applying these two inequalities, we first obtain a Hardy-type inequality within a domain characterized by a flat boundary (see Lemmas \ref{lem3} and \ref{lem2-p3}).  Then, by the usual patching technique using partition of unity, we prove the theorems for general bounded Lipschitz domains. A similar idea is used in \cite{adisahuroy}. In fact, they have established a full-range fractional version of C-K-N inequalities.

% \smallskip

%  The above theorems, for the special case
%  $\Phi(t)=t^p$, are established in \cite{adisahuroy}. 
% Their proof relies on the homogeneity of $t^p$. However, a general Orlicz function is not homogeneous. Thanks to Lemma \ref{Lem-H-p3}, it plays a crucial role in proving the results in this article.
\begin{rem} 
Let $(u)_\Omega$ be the  average of $u$ over $\Omega$, i.e.,  $(u)_\Omega=\frac{1}{|\Omega|}\int_\Omega u(x)\,dx,$
 where $|\Omega|$ is the Lebesgue measure of $\Omega.$ 
By applying fractional Orlicz Poincar\'e-Wirtinger inequality (Proposition \ref{poincare}), inequalities \eqref{2} and \eqref{1} are also can be restated as
\begin{equation*}
 \int_{\Omega}\Phi\left(\frac{|u(x)-(u)_\Omega|}{\delta^s_\Omega(x)}\right)\frac{dx}{\ln^{p^+_\Phi}\left(\frac{2R}{\delta_\Omega(x)}\right)}\leq C\int_{\Omega}\int_{\Omega}\Phi\left(|D_su(x,y)|\right)d\mu,\quad \forall\, u\in \mathcal{C}_c^1(\Omega)
\end{equation*}
and 
 \begin{equation*}
 \int_{\Omega}\Phi\left(\frac{|u(x)-(u)_\Omega|}{\delta^s_\Omega(x)}\right)dx \leq C\int_{\Omega}\int_{\Omega}\Phi\left(|D_su(x,y)|\right)d\mu,\quad \forall\,u\in \mathcal{C}_c^1(\Omega).
\end{equation*}
\end{rem}
In the next theorem, we consider the  cases when 
 $\Omega$ is in $\mathcal{A}_2$ and   $\mathcal{A}_3$. In these cases, we establish a fractional Hardy-type inequality with a logarithmic correction on the left-hand side of \eqref{1}. 
\begin{thm}\label{Orlicz-Hardy2*-p3} Let $N\ge 1,\,s\in (0,1),$ and $\Omega$ be an open set in $\mathbb{R}^N$. Let $R>0$ and  $\Omega_R:=\{x\in \Omega:\delta_\Omega(x)<R\}
 .$ For any Young function $\Phi\in \D,$ we have 
\begin{enumerate}[(i)]
    \item if $\Omega\in \mathcal{A}_2$ and $sp^+_\Phi=1,$ then for all $u\in \mathcal{C}_c^1(\Omega)$ with $supp(u)\subset\Omega_R$, \begin{equation}\label{3}
 \int_{\Omega_R}\Phi\left(\frac{|u(x)|}{\delta_\Omega^{s}(x)}\right)\frac{dx}{\ln^{p_\Phi^+}\left(\frac{2R}{\delta_\Omega(x)}\right)}\leq C\left(\int_{\Omega}\int_{\Omega}\Phi\left(|D_su(x,y)|\right)d\mu+\int_{\Omega_R}\Phi(|u(x)|)dx\right),
\end{equation}
\item if $\Omega\in \mathcal{A}_3$ and  $sp^-_\Phi=N$ with $N>1,$ then for all  $u\in \mathcal{C}_c^1(\Omega),$ 
\begin{equation}\label{4}
\int_{\Omega}\Phi\left(\frac{|u(x)|}{\delta_\Omega^{s}(x)}\right)\frac{dx}{\ln^{p_\Phi^+}\left(\max\left\{\frac{2R}{\delta_\Omega(x)},\frac{2\delta_\Omega(x)}{R}\right\}\right)}\leq C\left(\int_{\Omega}\int_{\Omega}\Phi\left(|D_su(x,y)|\right)d\mu+\int_{\Omega}\Phi(|u(x)|)dx\right),
\end{equation}
\end{enumerate}
where $C$ is a positive constant independent of $u$.
\end{thm}
 \begin{rem}
Let $\Phi(t) = t^p$ with $p > 1$. It is known (see for example \cite{Dyda2007}) that the boundary Hardy inequality \eqref{Hardy-p3} fails to hold in the critical cases: $sp = 1$ for $\Omega \in \mathcal{A}_2$ and $sp = N$ for $\Omega \in \mathcal{A}_3$. Thus, the above theorem provides a Hardy-type inequality with logarithmic correction in the critical cases. %For $\Phi(t)=t^p$ with $p>1,$ the above theorem is a particular case of Theorem 2 in \cite{adisahuroy}.
 \end{rem}
 
\smallskip

Next, we study a weighted fractional Orlicz boundary Hardy inequality for $\mathbb{R}^N_+.$ 
For $\Phi(t)=t^p$, $\alpha_1+\alpha_2\in (-1/p,s),$ and $(s-\alpha_1-\alpha_2)p\ne 1$, Dyda and Kijaczko \cite[Theorem 1]{Dyda2022}  proved the following weighted fractional boundary Hardy inequality:
  \begin{equation}\label{weight-p3}
\int_{\mathbb{R}^N_+}\Phi\left(\frac{|u(x)|}{x_N^{s-\alpha_1-\alpha_2}}\right) dx\leq C\int_{\mathbb{R}^N_+}\int_{\mathbb{R}^N_+}\Phi\left(|D_su(x,y)|x_N^{\alpha_1}y_N^{\alpha_2}\right)d\mu,\quad\forall\, u \in \mathcal{C}_c^1(\mathbb{R}^N_+),
\end{equation}
 where $C=C(s,p,\alpha_1,\alpha_2,N)>0.$ The unweighted ($\alpha_1=\alpha_2=0$) generalization of \cite[Theorem 1]{Dyda2022} in the Orlicz setting is established in \cite[$N=1$; Theorem 1.1]{salort2022} and 
  \cite[Theorem 1.5]{roy2022}. In particular, they proved that the above inequality holds under certain sufficient conditions on $s,\Phi$. For the weighted fractional Orlicz-Hardy inequality in $\mathbb{R}^N$ and $\mathbb{R}^N\setminus\{0\}$, we refer to \cite{subha2}.
 
 \smallskip
 
It is known that for $\Phi(t)=t^p$ and $\alpha_1=\alpha_2=0$, the above inequality  fails to hold in the critical
 case $(s-\alpha_1-\alpha_2)p^+_\Phi= 1$ (see, for example, \cite[Appendix]{adisahuroy}). In this case, the following theorem establishes a boundary Hardy inequality with a logarithmic correction on the left-hand side of \eqref{weight-p3}:
\begin{thm}\label{Orlicz-Hardy3-p3} Let $N\ge 1,\,s\in (0,1),\,R>0,\,\alpha_1,\alpha_2\in\mathbb{R}$, and $\Phi\in \D$ be a Young function. If  $(s-\alpha_1-\alpha_2)p^+_\Phi= 1$, then for every $u \in \mathcal{C}_c^1(\mathbb{R}^N_+)$ with $supp(u)\subset \mathbb{R}^{N-1}\times (0,R)$,
  \begin{equation*}
\int_{\mathbb{R}^N_+}\Phi\left(\frac{|u(x)|}{x_N^{s-\alpha_1-\alpha_2}}\right)\frac{dx}{\ln^{p^+_\Phi}\left(\frac{2R}{x_N}\right)} \leq C\int_{\mathbb{R}^N_+}\int_{\mathbb{R}^N_+}\Phi\left(|D_su(x,y)|x_N^{\alpha_1}y_N^{\alpha_2}\right)d\mu,
\end{equation*}
where $C$ is a positive constant independent of $u$.
\end{thm}
\begin{rem} It follows from  $(i)$ of Theorem \ref{Orlicz-Hardy2*-p3} that the inequality \eqref{3} holds for $\Omega=\mathbb{R}^N_+\in \mathcal{A}_2$ in the case $sp^+_\Phi=1$. Thus, the above theorem generalizes $(i)$ of Theorem \ref{Orlicz-Hardy2*-p3} for $\Omega = \mathbb{R}^N_+$. For $\Phi(t)=t^p$ with $p>1,$ the above theorem is a particular case of Theorem 6 in \cite{adisahuroy}.
\end{rem}

\smallskip

 The article is structured as follows: In section \ref{preli-p3}, we recall some properties of Young functions and Orlicz spaces and prove some essential lemmas frequently used in this article. In Section \ref{flatbunndary-p3}, we establish a fractional boundary Hardy-type inequality in a domain with flat boundaries and the complement of a ball.
 We present the proofs for  Theorem \ref{Orlicz-Hardy2-p3}, Theorem \ref{Orlicz-Hardy1*-p3}, Theorem \ref{Orlicz-Hardy2*-p3}, and Theorem \ref{Orlicz-Hardy3-p3} in Section \ref{orlicz-boundary-p3}. 
 
\section{preliminaries}\label{preli-p3}
In this section, we recall or prove some essential results we need to prove the main theorems of this article. Throughout this article, we shall use the following notations: 
\begin{itemize}
\item A point $x\in \mathbb{R}^N$ is represented as $x=(x^\prime,x_N)$ where $x^\prime=(x_1,x_2,\dots,x_{N-1})\in \mathbb{R}^{N-1}.$
    \item  $\mathcal{C}_c^1(\Omega)$ denotes the set of continuously differentiable functions with compact support.
    \item  For any $f,g:\Omega(\subset\mathbb{R}^N)\rightarrow \mathbb{R}$ we
denote $f\asymp g$ if there exist constants $C_1,C_2>0$ such that $C_1f(x)\leq g(x)\leq C_2f(x)$ for all $x\in \Omega.$ 
\end{itemize}

\subsection{Properties of Young functions} We recall the basic notions and some of the properties concerning Young functions satisfying the $\D$-condition. We refer to the books \cite{adams2003,Krasn1961} for a comprehensive
introduction to the subject.

%  \begin{defn}\label{defn-orlicz}
%     A  function  $\Phi:[0,\infty)\to [0,\infty)$ is called an Orlicz function if it has the following two properties:
%      \begin{enumerate}[(a)]
%          \item $\Phi$ can be represented in the form
%  \begin{equation*}
%      \Phi(t)=\int_0^t \varphi(s) ds\quad   \text{for}\,\, t\geq 0,
%  \end{equation*} 
%  where $\varphi$ is a non-decreasing right continuous function on $[0,\infty)$ satisfying $\varphi(0)=0,$ $\varphi(t)>0$ for $t>0,$ and $\lim_{t\to \infty}\varphi(t)=\infty$.
%         \item $\Phi$ satisfies the  $\D$-condition ($\Phi\in \D$), i.e., there exists a constant $C> 0$ such that 
% \begin{equation*}
%      \Phi(2t)\leq C\Phi(t),\quad \forall\,t\geq 0.
%  \end{equation*}
%     \end{enumerate}
%   \end{defn} 
\begin{lem}\cite{bonder2019}
    Let $\Phi$ be a Young function satisfying the $\D$-condition. Then, for every $a, b \geq 0,$ 
\begin{align}
 \Phi(a+b)&\leq 2^{p^+_\Phi }\left(\Phi(a)+\Phi(b)\right),\label{H2-p3}\\
     \Phi(ab)&\leq \max\{a^{p^-_\Phi },a^{p^+_\Phi }\}\Phi(b).\label{H3-p3}
\end{align}
\end{lem}

\begin{exmp}\label{example-p3} 
The following table lists the values of $p_\Phi^-$ and $p^+_\Phi$ for various Young functions. For the computations, we refer to \cite[Example 5.5]{subha2}.
 \smallskip
 
    \begin{center}
    \setlength{\arrayrulewidth}{0.5mm}
\setlength{\tabcolsep}{18pt}
\renewcommand{\arraystretch}{1.5}

\begin{tabular}{p{4cm}p{2cm}p{2cm}}
\hline

$\Phi(t)$ & $p^-_\Phi$  & $p^+_\Phi$  \\
\hline

 $t^p+t^q;\; q> p>1$  &$p$ & $q$
  \vspace{.3cm}\\ 
$\max\{t^p,t^q\};\; q> p>1$ &
 $p$ & $q$  \vspace{.3cm} \\

$t^p\ln (1+t);\;p\ge 1$ &  $p$ & $p+1$  \vspace{.3cm} \\

$(1+t)\ln(1+t)-t$  &$1$ & $2$  \vspace{.3cm} \\  

\hline
\end{tabular}
\vspace{.5cm}
\end{center}
For more examples of Young functions and their corresponding values of $p_\Phi^-$ and $p^+_\Phi$, see \cite{Mihai2012,salort-jde} and the references therein.
\end{exmp}

Next, we recall an important inequality for a Young function proved in \cite[Lemma 3.1]{subha2}, which plays a crucial role in the proof of all our theorems. For $\Phi(t)=t^p,$ the following lemma is established in \cite{adisahuroy, Nguyen2018} and is used to prove the fractional Caffarelli-Kohn-Nirenberg inequalities.
\begin{lem}\label{Lem-H-p3}
    Let $\Phi$ be a Young function that satisfies the $\D$-condition and $\Lambda>1$.
Then there exists $C=C(\Phi,\Lambda)>0$ such that 
     \begin{equation*}
 \Phi(a+b)\leq \lambda\Phi(a)+\frac{C}{(\lambda-1)^{p^+_\Phi-1}}\Phi(b),\;\;\;\forall\,a,b\in [0,\infty),\,\forall\,\lambda\in (1,\Lambda]. 
  \end{equation*}
  \end{lem}
\subsection{Some function spaces:}  Let $\Omega$ be an open set in $\mathbb{R}^N,$ and $\Phi\in \D$ be a Young function. 
\begin{enumerate}[(i)]
\item  \textbf{\underline{Orlicz spaces}:}
The  Orlicz space associated with $\Phi$ is  defined as 
$$ L^{\Phi}(\Omega)=\left\{u:\Omega\rightarrow \mathbb{R}\; \text{measurable}\; :\int_{\Omega}\Phi\left(|u(x)|\right)dx<\infty \right\}.$$ 
 The space $L^{\Phi}(\Omega)$  is a Banach space with respect to the following \textit{Luxemburg norm:}
\begin{equation*}
   \|u\|_{L^{\Phi}(\Omega)}=\inf\left\{ \lambda>0:\int_{\Omega}\Phi\left( \frac{|u(x)|}{\lambda}\right) dx\leq 1\right\}.
\end{equation*}

\item  \textbf{\underline{Fractional Orlicz-Sobolev spaces}:} Let $s\in (0,1)$. Then, the fractional Orlicz-Sobolev space is defined as
$$ W^{s,\Phi}(\Omega)=\left\{u\in L^\Phi(\Omega) :\; I_{\Phi,\Omega}(u)<\infty  \right\}, \quad  I_{\Phi,\Omega}(u):=\int_{\Omega}\int_{\Omega}\Phi\left( |D_su(x,y)|\right)d\mu .$$
The space $ W^{s,\Phi}(\Omega)$ is a Banach space with the norm $\|u\|_{W^{s,\Phi}(\Omega)}:=\|u\|_{L^\Phi(\Omega)}+[u]_{W^{s,\Phi}(\Omega)}$, where
\begin{equation*}
    [u]_{W^{s,\Phi}(\Omega)}=\inf\left\{ \lambda>0: I_{\Phi,\Omega}\left(\frac{u}{\lambda}\right)\leq 1\right\}.
\end{equation*}
\end{enumerate}
For details on Orlicz and fractional Orlicz-Sobolev spaces, we refer to \cite{adams2003, bonder2019, Krasn1961}.
\smallskip

Now, we recall the {\it{fractional Poincar\'e-Wirtinger inequality}} for a Young function $\Phi\in \D$ on  $\Omega_\lambda:=\{\lambda x:x\in \Omega\}$, where  $\Omega$ is a bounded open set in $\mathbb{R}^N$ and $\lambda>0$. 
\begin{prop}\cite[Proposition 2.3]{subha2}\label{poincare}
Let $N\geq 1,\,s\in (0,1)$, $\lambda>0$, and $\Omega$ be a bounded open set in $\mathbb{R}^N$. Then for any Young function $\Phi\in \D$, there exists $C=C(s,N,\Omega,\Phi)>0$  so that
\begin{equation*}
  \int_{\Omega_\lambda}\Phi(|u(x)-(u)_{\Omega_\lambda}|)dx\leq C\int_{\Omega_\lambda}\int_{\Omega_\lambda}\Phi\left(\lambda^s|D_su(x,y)|\right)d\mu,\quad \forall\,u\in W^{s,\Phi}(\Omega_\lambda).
\end{equation*}
\end{prop}

The following lemma proves that if $u\in W^{s,\Phi}(\Omega)$ and $\xi$ is a test function, then $\xi u\in W^{s,\Phi}(\Omega)$. This lemma plays an important role in the establishment of Theorem \ref{Orlicz-Hardy2-p3}, Theorem \ref{Orlicz-Hardy1*-p3}, and Theorem \ref{Orlicz-Hardy2*-p3}. For $\Phi(t)=t^p,$  this lemma is established in \cite[Lemma 5.3]{Hitchhiker2012}. We adapt their proof to our case. 

\begin{lem}\label{lem-ext}
    Let $\Omega$ be an open subset of $\mathbb{R}^N,\,N\ge 1$, $s\in (0,1),$ and $\xi\in \mathcal{C}^{0,1}(\Omega)$ such that $0\le \xi\le 1$. Then for any  Young function $\Phi\in \D$, there exists $C=C(s,N,\Phi,\Omega)>0$ such that
    \begin{multline*}
\int_\Omega\Phi(|(\xi u)(x)|)dx+ \int_{\Omega}\int_{\Omega}\Phi\left(|D_s(\xi u)(x,y)|\right)d\mu\leq C\int_\Omega\Phi(|u(x)|)dx\\+C\int_{\Omega}\int_{\Omega}\Phi\left(|D_su(x,y)|\right)d\mu,\quad \forall\,u\in W^{s,\Phi}(\Omega).
    \end{multline*}
    \end{lem}
    \begin{proof}
 Let $u\in W^{s,\Phi}(\Omega)$. Since $\Phi$ is an increasing function and $\xi(x)\le 1$, we have $$\int_\Omega\Phi(|(\xi u)(x)|)dx\le \int_\Omega\Phi(| u(x)|)dx.$$ Adding and subtracting the factor $\xi(x)u(y)$ and using \eqref{H2-p3}, we get 
    \begin{multline*}
\int_{\Omega}\int_{\Omega}\Phi\left(\frac{|u(x)-\xi(y)u(y)|}{|x-y|^s}\right)d\mu\le 2^{p^+_\Phi }  \int_{\Omega}\int_{\Omega}\Phi\left(\frac{\xi(x)|u(x)-u(y)|}{|x-y|^s}\right)d\mu\\+2^{p^+_\Phi } \int_{\Omega}\int_{\Omega}\Phi\left(\frac{|\xi(x)u(y)-\xi(y)u(y)|}{|x-y|^s}\right)d\mu.
    \end{multline*}
Next, we estimate the second term on the right-hand side of the above inequality. Since $\xi\in\mathcal{C}^{0,1}(\Omega)$ and $0\le \xi\le 1,$ we have   
 \begin{multline*}
\int_{\Omega}\int_{\Omega}\Phi\left(\frac{|\xi(x)u(y)-\xi(y)u(y)|}{|x-y|^s}\right)d\mu\le \int_{\Omega}\int_{\Omega\cap |x-y|\le 1}\Phi\left(\frac{C|x-y||u(y)|}{|x-y|^s}\right)d\mu\\+\int_{\Omega}\int_{\Omega\cap |x-y|>1}\Phi\left(\frac{|u(y)|}{|x-y|^s}\right)d\mu.
 \end{multline*}
Now use \eqref{H3-p3} to obtain
\begin{multline*}
    \int_{\Omega}\int_{\Omega\cap |x-y|\le 1}\Phi\left(\frac{C|x-y||u(y)|}{|x-y|^s}\right)d\mu\leq C  \int_{\Omega}\int_{\Omega\cap |x-y|\le  1}\frac{\Phi\left(|u(y)|\right)}{|x-y|^{(s-1)p^-_\Phi }}d\mu\\=C  \int_{\Omega}\Phi\left(|u(y)|\right)dy\int_{\Omega\cap |x-y|\le  1}\frac{dx}{|x-y|^{(s-1)p^-_\Phi+N }}\le C\int_{\Omega}\Phi\left(|u(y)|\right)dy
\end{multline*}
where the last inequality follows from the fact that the kernel $|x-y|^{-(s-1)p^-_\Phi-N}$ is
summable with respect to $x$ if $|x-y|\le  1$ since $(s-1)p^-_\Phi+N<N$. Moreover, the
 kernel $|x-y|^{-N-sp^-_\Phi }$ is also summable when $|x-y|>1$ since $N+sp^-_\Phi > N.$ Thus, using \eqref{H3-p3} again we get
 \begin{equation*}
     \int_{\Omega}\int_{\Omega\cap |x-y|>1}\Phi\left(\frac{|u(y)|}{|x-y|^s}\right)d\mu\le \int_{\Omega}\int_{\Omega\cap |x-y|> 1}\frac{\Phi\left(|u(y)|\right)}{|x-y|^{N+sp^-_\Phi }}dxdy\le C\int_{\Omega}\Phi\left(|u(y)|\right)dy.
 \end{equation*}
 Finally, by combining the above inequalities, we conclude our result.
    \end{proof}
The next lemma establishes a relationship between the averages of $u$ over two disjoint sets.
\begin{lem}\label{Disjointset} Let $P$ and $Q$ be two disjoint subsets of $\mathbb{R}^N$. Then for every Young function $\Phi\in \D$, we have %, there exists a constant $C=C(\Phi)>0$ such that
    $$\Phi\left(|(u)_{P}-(u)_{Q}|\right)\le \frac{2^{p^+_\Phi}}{\min\{|P|,|Q|\}}\int_{P\cup Q}\Phi(|u(x)-(u)_{P\cup Q}|)dx,\quad \forall\,u\in \mathcal{C}_c^1(\mathbb{R}^N).$$
    \end{lem}
    \begin{proof}
     Let $u\in \mathcal{C}_c^1(\mathbb{R}^N)$.  By using  \eqref{H2-p3} and Jensen's inequality,  
we get
\begin{align*}
     \Phi\left(|(u)_P-(u)_Q |\right)=&\Phi\left(|(u)_P-(u)_{P\cup Q}+(u)_Q+(u)_{P\cup Q}|\right)\\\le& 2^{p^+_\Phi }\Phi\left(|(u)_P-(u)_{P\cup Q}|\right)+2^{p^+_\Phi }\Phi\left(|(u)_Q-(u)_{P\cup Q}|\right)\\\le &2^{p^+_\Phi }\Phi\left( \frac{1}{|P|}\int_{P}|u(x)-(u)_{P\cup Q}|dx\right)+2^{p^+_\Phi }\Phi\left(\frac{1}{|Q|}\int_{ Q}|u(x)-(u)_{P\cup Q}|dx\right)\nonumber\\\leq & \frac{2^{p^+_\Phi }}{|P|}\int_P\Phi(|u(x)-(u)_{P\cup Q}|)dx+\frac{2^{p^+_\Phi }}{|Q|}\int_{ Q}\Phi(|u(x)-(u)_{P\cup Q}|)dx\nonumber\\\leq &  \frac{2^{p^+_\Phi }}{\min\{|P|,|Q|\}}\int_{P\cup Q}\Phi(|u(x)-(u)_{P\cup Q}|)dx.
\end{align*}
This completes the proof.
\end{proof}

Next, we recall a
lemma for a domain above the graph of a Lipschitz function. This lemma is useful for proving Theorem \ref{Orlicz-Hardy2-p3}, Theorem \ref{Orlicz-Hardy1*-p3}, and Theorem \ref{Orlicz-Hardy2*-p3}. For a proof, we refer to \cite[Appendix]{adisahuroy}.
\begin{lem}\label{lips}
    Let $\Omega$ be a domain above the graph of the Lipschitz function $\gamma:\mathbb{R}^{N-1}\rightarrow\mathbb{R}$.  
    Then, there exist positive constants $C_1$ and $C_2$ depending only on $\gamma$ such that $$C_1\delta_\Omega(x)\le x_N-\gamma(x^\prime)\le C_2\delta_\Omega(x),\quad\forall x\in \Omega.$$
\end{lem}
\begin{lem}[Change of variable formula]\cite[Theorem 3.9]{Evans-L2015}\label{changeofvar} 
 Let $T:\mathbb{R}^N\rightarrow\mathbb{R}^N$ be an invertible Lipschitz function and $DT(x)$ be the gradient matrix of $T$. If $|det DT(x)|=1$, then for every
measurable function $u:\mathbb{R}^N\rightarrow \mathbb{R}$ with $supp(u)\subset\Omega\subset\mathbb{R}^N,$ $$\int_\Omega u(x)dx=\int_{T(\Omega)}u\circ T^{-1}(x)dx.$$ 
\end{lem}

\section{Orlicz Hardy inequality in a domain with flat boundaries}\label{flatbunndary-p3}
In this section, we establish boundary Hardy-type inequalities within a domain characterized by a flat boundary and within the complement of a ball in $\mathbb{R}^N.$ These results will be used to prove our main theorems. 

\smallskip 
Let $2n\ge T>0$ such that $\frac{2n}{T}\in\mathbb{N}.$ Define the cube $$\Omega_{n,T}=(-n,n)^{N-1}\times (0,T)$$ 
and for each $k\in \mathbb{Z}$ with $k\le 0$, set $$A_k:=\left\{(x^\prime,x_N):x^\prime \in (-n,n)^{N-1},\;2^kT\le x_N<2^{k+1}T\right\}.$$ Then we have $$\Omega_{n,T}=\bigcup_{k=-\infty}^{-1}A_k.$$ 
Further, we divide each $A_k$ into disjoint cubes, each of side length $2^kT$, say $A_k^i$. There are $2^{(-k+1)(N-1)}n^{N-1}T^{1-N}$ such cubes $A^i_k$.  
For simplicity, let $\sigma_k=2^{(-k+1)(N-1)}n^{N-1}T^{1-N}.$
 Then, 
$$A_k=\bigcup_{i=1}^{\sigma_k}A_k^i.$$
  %Then,  $A_k=\bigcup_{i=1}^{\sigma_k}A_k^i.$ 
  Here,  $\sigma_k$ is a natural number for $k\in \mathbb{Z}^-\cup\{0\}$, since $ \frac{2n}{T}\in\mathbb{N}.$

\smallskip

The subsequent lemma establishes an inequality for every $A_k^i$, which will be useful in proving Hardy-type inequalities within the domain $\Omega_{n,T}$ (see Lemma \ref{lem3} and Lemma \ref{lem2-p3}). 
\begin{lem}\label{lemm-1} Let $N\ge 1,\,s\in (0,1),\,\alpha_1,\alpha_2\in \mathbb{R}$, and $\tau:=s-\alpha_1-\alpha_2$. Then for any Young function $\Phi\in \D$, there exists $C=C(s,\alpha_1,\alpha_2,N,T,\Phi)>0$ so that for every $u\in \mathcal{C}_c^1(\Omega_{n,T}),$
    \begin{equation*}
        \int_{A_k^i}\Phi\left(\frac{|u(x)|}{x_N^\tau}\right)dx \leq C2^{kN}\Phi\left(2^{-k\tau}|(u)_{A_k^i}|\right)+C\int_{A_k^i}\int_{A_k^i}\Phi\left(|D_su(x,y)|x_N^{\alpha_1}y_N^{\alpha_2}\right)d\mu.
    \end{equation*}
\begin{proof}
Let $u\in \mathcal{C}_c^1(\Omega_{n,T}).$ Notice that $A_k^i$ is a translation of $(2^kT,2^{k+1}T)^N$.  Thus, applying Proposition \ref{poincare} with $\Omega=(T,2T)^N$ and $\lambda=2^k$, and using translation invariance, we obtain 
    \begin{equation*}
        \int_{A_k^i} \Phi\left(|u(x)-(u)_{A_k^i}|\right)dx\leq C\int_{A_k^i}\int_{A_k^i}\Phi\left(2^{ks}|D_s u(x,y)|\right)d\mu,
    \end{equation*}
    where $C=C(s,N,T,\Phi)$ is a positive constant. Now using \eqref{H2-p3}, we obtain
    \begin{align*}
    \int_{A_k^i}\Phi\left(|u(x)|\right) dx& =  \int_{A_k^i}\Phi\left(|(u)_{A_k^i}+u(x)-(u)_{A_k^i}|\right)dx\\& \leq  2^{p^+_\Phi }\int_{A_k^i}\Phi\left(|(u)_{A_k^i}|\right)dx+ 2^{p^+_\Phi }\int_{A_k^i}\Phi\left(|u(x)-(u)_{A_k^i}|\right)dx\\& \leq  2^{p^+_\Phi }|A^i_k|\Phi\left(|(u)_{A_k^i}|\right)+ 2^{p^+_\Phi }C\int_{A_k^i}\int_{A_k^i}\Phi\left(2^{ks}|D_s u(x,y)|\right)d\mu.
\end{align*}   
    Replace $u$ with $2^{-k\tau}u$ in the above inequality to get
        \begin{equation*}
        \int_{A_k^i} \Phi\left(2^{-k\tau}|u(x)|\right)dx\leq C2^{kN}\Phi\left(2^{-k\tau}|(u)_{A_k^i}|\right)+C\int_{A_k^i}\int_{A_k^i}\Phi\left(2^{k(s-\tau)}|D_s u(x,y)|\right)d\mu.
    \end{equation*}
For $x=(x^\prime,x_N)\in A_k^i,$ we have $2^kT\le x_N< 2^{k+1}T$. Thus, using \eqref{H3-p3}, we get  $C=C(s,\alpha_1,\alpha_2,T,\Phi)>0$ such that for every $x,y\in A_k^i$, 
% Thus, using \eqref{H3-p3}, we get $C=C(s,\alpha_1,\alpha_2,\Phi)>0$ such that for every $x,y\in A_k^i$,
$$\Phi\left(\frac{|u(x)|}{x_N^\tau}\right)\le C\Phi(2^{-k\tau}|u(x)|)$$ and 
$$\Phi\left(2^{k(s-\tau)}|D_s u(x,y)|\right)=\Phi\left(2^{k(\alpha_1+\alpha_2)}|D_s u(x,y)|\right)\le C\Phi\left(x_N^{\alpha_1}y_N^{\alpha_2}|D_s u(x,y)|\right).$$ 
Hence, the conclusion follows from the last three inequalities.
\end{proof}
    
\end{lem}
The next lemma relates the averages of $u$ over two disjoint cubes $A_k^i$ and $A_{k+1}^j$, where $A_{k+1}^j$ ($j\in \{1,\dots,\sigma_{k+1}\}$) is chosen so that $A_k^i$ lies below it. There are $2^{N-1}$ cubes $A_k^i$ lying below $A^j_{k+1}$, since $\sigma_k=2^{N-1}\sigma_{k+1}$. For $\Phi(t)=t^p$, this lemma is proved in \cite[Lemma 3.2]{adisahuroy}. We adapt their proof in the Orlicz setting.

\begin{lem}\label{Lem-avg}
Let $N\ge 1,\,s\in (0,1),\,\alpha_1,\alpha_2\in \mathbb{R}$, and $\tau:=s-\alpha_1-\alpha_2$. Let $A_{k+1}^j$ be a cube such that  $A_k^i$ lies below the cube $A_{k+1}^j$. Then for any Young function $\Phi\in \D$, there exists  a positive constant $C$ independent of both $u$ and $k$ so that  
    $$\Phi\left(2^{-k\tau}|(u)_{A_k^i}-(u)_{A_{k+1}^j}|\right)\le \frac{C}{2^{kN}}\int_{A_k^i\cup A_{k+1}^j}\int_{A_k^i\cup A_{k+1}^j}\Phi\left(|D_su(x,y)|x_N^{\alpha_1}y_N^{\alpha_2}\right)d\mu,\; \forall\,u\in \mathcal{C}_c^1(\Omega_{n,T}).$$
    \end{lem}
    \begin{proof}  Let $u\in \mathcal{C}_c^1(\Omega_{n,T})$.
   By applying Lemma \ref{Disjointset} with $P=A_k^i$ and $Q=A^j_{k+1}$,
we get
$$\Phi\left(|(u)_{A_k^i}-(u)_{A_{k+1}^j} |\right)\leq   \frac{C}{2^{kN}}\int_{A_k^i\cup A_{k+1}^j}\Phi\left(|u(x)-(u)_{A^i_k\cup A^j_{k+1}}|\right)dx.$$
Choose an open set $\Omega$ such that $\Omega_\lambda$ with $\lambda=2^{k+1}$ is a translation of $A_k^i\cup A_{k+1}^j$.  Applying Proposition \ref{poincare} with this $\Omega$ and $\lambda=2^{k+1}$, and using translation invariance, we obtain
\begin{equation*}
    \int_{A_k^i\cup A^j_{k+1}}\Phi\left(|u(x)-(u)_{A^i_k\cup A^j_{k+1}}|\right)dx\leq C \int_{A^i_k\cup A^j_{k+1}}\int_{A^i_k\cup A^j_{k+1}}\Phi\left(2^{ks+s}|D_su(x,y)|\right)d\mu.
\end{equation*}
%where $C=C(s,N,R,\Phi)$ is a positive constant. 
Combining the above two inequalities (and then  replacing $u$ with $2^{-k\tau}u$), we get 
$$\Phi\left(2^{-k\tau}|(u)_{A_k^i}-(u)_{A_{k+1}^j} |\right)\leq  \frac{C}{2^{kN}} \int_{A^i_k\cup A^j_{k+1}}\int_{A^i_k\cup A^j_{k+1}}\Phi\left(2^{k(s-\tau)+s}|D_su(x,y)|\right)d\mu.$$
For $x=(x^\prime,x_N)\in A_k^i\cup A^j_{k+1},$ we have $2^kT\le x_N< 2^{k+2}T$. Thus, 
using \eqref{H3-p3}, we get $C=C(s,\alpha_1,\alpha_2,T,\Phi)>0$ such that for every $x,y\in A_k^i\cup A^j_{k+1}$, 
$$\Phi\left(2^{k(s-\tau)+s}|D_s u(x,y)|\right)=\Phi\left(2^{k(\alpha_1+\alpha_2)+s}|D_s u(x,y)|\right)\le C\Phi\left(x_N^{\alpha_1}y_N^{\alpha_2}|D_s u(x,y)|\right).$$
Hence, the conclusion follows from the above two inequalities.
    \end{proof}
    
    \smallskip

Next, we establish a boundary Hardy-type inequality for the domain $\Omega_{n,T}$ by using the lemma Lemma \ref{Lem-H-p3}.  This lemma is crucial for proving Theorem \ref{Orlicz-Hardy1*-p3}, as the proof for a general Lipschitz domain can be derived using the standard patching technique with a partition of unity. Details of this method will be provided in the next section.
\begin{lem}\label{lem3}
    Let $N\ge 1,\,s\in (0,1),$ and $\Phi\in \D$ be a Young function. If $sp^+_\Phi < 1$, then there exists  $C=C(s,N,\Phi,T)>0$ so that  $$\int_{\Omega_{n,T}}\Phi\left(\frac{|u(x)|}{x_N^s}\right)dx\le C\int_{\Omega_{n,T}}\int_{\Omega_{n,T}}\Phi(|D_su(x,y)|)d\mu+C\int_{\Omega_{n,T}} \Phi(|u(x)|)dx,\quad\forall\,u\in \mathcal{C}_c^1(\Omega_{n,T}).$$
    \end{lem}
    \begin{proof}
Let $u\in \mathcal{C}_c^1(\Omega_{n,T}).$ Applying Lemma \ref{lemm-1} with $\alpha_1=\alpha_2=0$, we obtain
\begin{align*}
    \int_{A_k^i}\Phi\left(\frac{|u(x)|}{x_N^s}\right)dx \le C2^{kN}\Phi\left(2^{-ks}|(u)_{A_k^i}|\right)+C\int_{A_k^i}\int_{A_k^i} \Phi(|D_su(x,y)|)d\mu.
\end{align*}       
Sum up the above inequalities from $i=1$ to $\sigma_k$ to obtain  
  $$\int_{A_k}\Phi\left(\frac{|u(x)|}{x_N^s}\right)dx\le C2^{kN}\sum_{i=1}^{\sigma_k}\Phi\left(2^{-ks}|(u)_{A_k^i}|\right)+C\int_{A_k}\int_{A_k} \Phi(|D_su(x,y)|)d\mu.$$
Let $m\in \mathbb{Z}$ such that $m<-2$, and for simplicity, let $a_k=\sum_{i=1}^{\sigma_{k}}\Phi\left(2^{-ks}|(u)_{A_k^i}|\right).$ Then, summing the above inequalities from $k=m$ to $-1$, we get 
\begin{equation}\label{lem1-eq1-3-p3}
\sum_{k=m}^{-1}\int_{A_k}\Phi\left(\frac{|u(x)|}{x_N^s}\right)dx\le C\sum_{k=m}^{-1}2^{kN}a_k+C\int_{\Omega_{n,T}}\int_{\Omega_{n,T}} \Phi(|D_su(x,y)|)d\mu.
\end{equation}
Next, we estimate the first term on the right-hand side of \eqref{lem1-eq1-3-p3}. Let $A_{k+1}^j$ be a cube such that $A_k^i$ lies below it, and $\Lambda>1$ (which will be determined later). By applying the triangle inequality and  Lemma \ref{Lem-H-p3} with $\lambda=\Lambda$, we get $C=C(\Phi,\Lambda)>0$ such that 
 \begin{align*}
 \Phi\left(2^{-ks}|(u)_{A_k^i}|\right)&\leq\Phi \left(2^{-ks}|(u)_{A_{k+1}^j}|+2^{-ks}|(u)_{A_k^i}- (u)_{A_{k+1}^j}|\right)\nonumber\\&\leq \Lambda\Phi\left(2^{-ks}|(u)_{A_{k+1}^j}|\right)+C\Phi\left(2^{-ks}|(u)_{A_k^i}- (u)_{A_{k+1}^j}|\right).
 \end{align*}
Further, we use \eqref{H3-p3} to estimate the first term on the right hand of the above inequality:
 \begin{align*}
     \Phi\left(2^{-ks}|(u)_{A_{k+1}^j}|\right)=\Phi\left(2^s\,2^{-(k+1)s}|(u)_{A_{k+1}^j}|\right)\leq 2^{sp_\Phi^+}\Phi\left(2^{ -(k+1)s}|(u)_{A_{k+1}^j}|\right).
\end{align*} 
By combining the above two inequalities and using Lemma \ref{Lem-avg} with $\alpha_1=\alpha_2=0$, we obtain 
 \begin{equation*}
     \Phi\left(2^{-ks}|(u)_{A_k^i}|\right)\leq 2^{sp^+_\Phi}\Lambda\Phi\left(2^{-(k+1)s}|(u)_{A_{k+1}^j}|\right) + \frac{C}{2^{kN}}\int_{A_k^i\cup A_{k+1}^j}\int_{A_k^i\cup A_{k+1}^j}\Phi(|D_su(x,y)|)d\mu.
 \end{equation*}
Notice that, there are $2^{N-1}$ such cubes $A_k^i$ lying below the cube $A_{k+1}^j$. Therefore, multiplying the above inequality by $2^{kN}$ and summing from $i=2^{N-1}(j-1)+1$ to $2^{N-1}j$, we can derive 
 \begin{multline*}
2^{kN} \sum^{2^{N-1}j}_{i=2^{N-1}(j-1)+1} \Phi\left(2^{-ks}|(u)_{A_k^i}|\right)\leq 2^{N-1}2^{(kN+sp^+_\Phi)}\Lambda\Phi\left(2^{ -(k+1)s}|(u)_{A_{k+1}^j}|\right)\\+C\sum^{2^{N-1}j}_{i=2^{N-1}(j-1)+1}\int_{A_k^i\cup A_{k+1}^j}\int_{A_k^i\cup A_{k+1}^j}\Phi(|D_su(x,y)|)d\mu.
 \end{multline*}
 Again, summing the above inequalities from $j=1$ to $\sigma_{k+1}$ and using the fact that 
 \begin{equation}\label{eqal-p3}
\sum_{j=1}^{\sigma_{k+1}}\left(\sum^{2^{N-1}j}_{i=2^{N-1}(j-1)+1} \Phi\left(2^{-ks}|(u)_{A_k^i}|\right)\right)=\sum_{i=1}^{\sigma_{k}}\Phi\left(2^{-ks}|(u)_{A_k^i}|\right)
 \end{equation}
 and
%  we obtain 
%  \begin{multline*}
% 2^{kN}\sum_{i=1}^{\sigma_{k}}\Phi\left(2^{-ks}|(u)_{A_k^i}|\right)\leq 2^{sp^+_\Phi-1}\Lambda 2^{N(k+1)}\sum_{j=1}^{\sigma_{k+1}}\Phi\left(2^{ -(k+1)s}|(u)_{A_{k+1}^j}|\right)\\+C\sum_{j=1}^{\sigma_{k+1}}\left(\sum^{2^{N-1}j}_{i=2^{N-1}(j-1)+1}\int_{A_k^i\cup A_{k+1}^j}\int_{A_k^i\cup A_{k+1}^j}\Phi(|D_su(x,y)|)d\mu\right).
%  \end{multline*}
% Moreover,
\begin{equation}\label{eqn2-lem1-3-p3}
\sum_{j=1}^{\sigma_{k+1}}\left(\sum^{2^{N-1}j}_{i=2^{N-1}(j-1)+1}\int_{A_k^i\cup A_{k+1}^j}\int_{A_k^i\cup A_{k+1}^j}\Phi(|D_su|)d\mu\right)\leq \int_{A_k\cup A_{k+1}}\int_{A_k\cup A_{k+1}}\Phi(|D_su|)d\mu,
 \end{equation}
we obtain % Then, the above two inequalities yield
\begin{equation*}
    2^{kN}a_k\leq 2^{N-1+kN+sp^+_\Phi} \Lambda \,a_{k+1}+C\int_{A_k\cup A_{k+1}}\int_{A_k\cup A_{k+1}}\Phi(|D_su(x,y)|)d\mu.
\end{equation*}
Summing the above inequalities from $k=m$ to $-2,$ we obtain
\begin{equation*}
    \sum_{k=m}^{-2} 2^{kN}a_k\leq 2^{sp^+_\Phi-1}\Lambda\sum_{k=m}^{-2} 2^{N(k+1)}a_{k+1}+C\sum_{k=m}^{-2} \int_{A_k\cup A_{k+1}}\int_{A_k\cup A_{k+1}}\Phi(|D_su(x,y)|)d\mu.
\end{equation*}
By changing sides, re-indexing, and rearranging, we get   
\begin{equation*}
 \frac{\Lambda 2^{mN}a_m}{2^{-sp^+_\Phi+1}} +\left(1-2^{sp^+_\Phi-1}\Lambda\right)\sum_{k=m}^{-1} 2^{kN}a_{k}\le \frac{a_{-1}}{2^N}+C\sum_{k=m}^{-2} \int_{A_k\cup A_{k+1}}\int_{A_k\cup A_{k+1}}\Phi(|D_su(x,y)|)d\mu.
\end{equation*}
Since $sp^+_\Phi<1,$ there exists $\Lambda=\Lambda(s,p^+_\Phi)>1$ such that $$2^{sp^+_\Phi-1}\Lambda<1.$$
Further, use Jensen's inequality and \eqref{H3-p3} to get
\begin{align*}
   a_{-1}=   \sum_{i=1}^{\sigma_{-1}}\Phi\left(2^{s}|(u)_{A_{-1}^i}|\right)\le \sum_{i=1}^{\sigma_{-1}}\frac{1}{|A_{-1}^i|}\int_{{A_{-1}^i}}\Phi(2^{s}|u(x)|)dx\le  C\int_{\Omega_{n,T}}\Phi(|u(x)|)dx.
\end{align*}
% Thus, substituting the value of $a_k$ into the above inequality and 
Consequently, using $\Lambda a_m > 0$, we obtain 
 \begin{equation*}
   \sum_{k=m}^{-1} 2^{kN}a_k\le C\int_{\Omega_{n,T}}\Phi(|u(x)|)dx+C \int_{\Omega_{n,T}}\int_{\Omega_{n,T}}\Phi(|D_su(x,y)|)d\mu.
 \end{equation*}
Hence the result follows from \eqref{lem1-eq1-3-p3} by taking $m\to -\infty$. 
\end{proof}

The following lemma is helpful to prove Theorem \ref{Orlicz-Hardy2-p3} and $(i)$ of Theorem \ref{Orlicz-Hardy2*-p3}. It also establishes $(i)$ of Theorem \ref{Orlicz-Hardy2*-p3} when $\Omega=\mathbb{R}^N_+$ with test functions supported on cubes.

\begin{lem}\label{lem2-p3}
    Let $N\ge 1,s\in (0,1)$, and $\alpha_1,\alpha_2,\tau\in\mathbb{R}$ be such that $\tau=s-\alpha_1-\alpha_2$. For any Young function $\Phi\in \D$, if $\tau p^+_\Phi =1$, then there exists  $C=C(s,N,\Phi,T)>0$ so that for every $u\in \mathcal{C}_c^1(\Omega_{n,T}),$ $$\int_{\Omega_{n,T}}\Phi\left(\frac{|u(x)|}{x_N^\tau}\right)\frac{dx}{\ln^{p^+_\Phi}(2T/x_N)}\le C\left(\int_{\Omega_{n,T}}\int_{\Omega_{n,T}}\Phi(|D_su|x_N^{\alpha_1}y_N^{\alpha_2})d\mu+ \int_{\Omega_{n,T}} \Phi(|u(x)|)dx\right).$$
    \end{lem}
    \begin{proof}
For each $x_N\in A_k^i,$ we have $ x_N<2^{k+1}T
$ which implies $\ln \left(2T/x_N\right)>(-k)\ln 2$ for all $k\in\mathbb{Z}^-.$ Now, applying Lemma \ref{lemm-1} and using $(-k)^{p^+_\Phi }\ge 1$ for all $k\in\mathbb{Z}^-$, we obtain
\begin{multline*}
    \int_{A_k^i}\Phi\left(\frac{|u(x)|}{x_N^\tau}\right)\frac{dx}{\ln^{p^+_\Phi}(2T/x_N)} \le \frac{1}{(-k\ln 2)^{p^+_\Phi}} \int_{A_k^i}\Phi\left(\frac{|u(x)|}{x_N^\tau}\right)dx\\\le \frac{C2^{kN}}{(-k)^{p^+_\Phi }}\Phi\left(2^{-k\tau}|(u)_{A_k^i}|\right)+C\int_{A_k^i}\int_{A_k^i} \Phi\left(|D_su(x,y)|x_N^{\alpha_1}y_N^{\alpha_2}\right)d\mu.
\end{multline*}  
 Summing the above inequalities from $i=1$ to $\sigma_k,$ we obtain  
\begin{multline*}
    \int_{A_k}\Phi\left(\frac{|u(x)|}{x_N^\tau}\right)\frac{dx}{\ln^{p^+_\Phi}(2T/x_N)}\le \frac{C2^{kN}}{(-k)^{p^+_\Phi }}\sum_{i=1}^{\sigma_k}\Phi\left(2^{-k\tau}|(u)_{A_k^i}|\right)\\+C\int_{A_k}\int_{A_k} \Phi\left(|D_su(x,y)|x_N^{\alpha_1}y_N^{\alpha_2}\right)d\mu.
\end{multline*}
Let $m\in\mathbb{Z}$ such that $m\le -2,$ and for simplicity, let $a_k=\sum_{i=1}^{\sigma_{k}}\Phi\left(2^{-k\tau}|(u)_{A_k^i}|\right).$ Then, summing the above inequalities from $k=m$ to $-1$, we get 
\begin{equation}\label{thm3-eq1-p3}
\sum_{k=m}^{-1}\int_{A_k}\Phi\left(\frac{|u(x)|}{x_N^\tau}\right)\frac{dx}{\ln^{p^+_\Phi}(2T/x_N)}\le C\sum_{k=m}^{-1}\frac{2^{kN}}{(-k)^{p^+_\Phi }}a_k+C\sum_{k=m}^{-1}\int_{A_k}\int_{A_k} \Phi\left(|D_su(x,y)|x_N^{\alpha_1}y_N^{\alpha_2}\right)d\mu.
\end{equation}
Next, we estimate the first term on the right-hand side of \eqref{thm3-eq1-p3}. Let $A_{k+1}^j$ be a cube such that $A_k^i$ lies below it. Then, by using the triangular inequality and Lemma \ref{Lem-H-p3} with $\Lambda=2^{p^+_\Phi }$, we get  $C=C(\Phi,\Lambda)>0$ satisfying 
 \begin{align*}
 \Phi\left(2^{-k\tau}|(u)_{A_k^i}|\right)&\leq\Phi \left(2^{-k\tau}|(u)_{A_{k+1}^j}|+2^{-k\tau}\left|(u)_{A_k^i}- (u)_{A_{k+1}^j}\right|\right)\\&\leq \lambda\Phi\left(2^{-k\tau}|(u)_{A_{k+1}^j}|\right)+\frac{C}{\left(\lambda-1\right)^{p^+_\Phi -1}}\Phi\left(2^{-k\tau}\left|(u)_{A_k^i}- (u)_{A_{k+1}^j}\right|\right),
 \end{align*}
 for every $\lambda\in(1,2^{p^+_\Phi })$. Now we use
 \eqref{H3-p3} and $\tau p^+_\Phi= 1$
 to yield
 \begin{align*}
     \Phi\left(2^{-k\tau}|(u)_{A_{k+1}^j}|\right)\le \max\{2^{\tau p^-_\Phi},2^{\tau p_\Phi^+}\}\Phi\left(2^{-(k+1)\tau}|(u)_{A_{k+1}^j}|\right)\leq 2\Phi\left(2^{-(k+1)\tau}|(u)_{A_{k+1}^j}|\right).
\end{align*} 
 By combining the above two inequalities and using Lemma \ref{Lem-avg}, we obtain
 \begin{multline}\label{thm3-log-eqn4*}
 \Phi\left(2^{-k\tau}|(u)_{A_k^i}|\right)\leq 2\lambda\Phi\left(2^{-(k+1)\tau}|(u)_{A_{k+1}^j}|\right) \\+ \frac{C}{2^{kN}\left(\lambda-1\right)^{p^+_\Phi -1}}\int_{A_k^i\cup A_{k+1}^j}\int_{A_k^i\cup A_{k+1}^j}\Phi\left(|D_su(x,y)|x_N^{\alpha_1}y_N^{\alpha_2}\right)d\mu.
 \end{multline}
Now, for each $k\in \mathbb{Z}^-$, we choose 
  $$\lambda(k):=\left(\frac{-k}{-k-1/2}\right)^{p^+_\Phi -1}.$$ For these choices of $\lambda$, we can verify that $$\lambda\in (1,2^{p^+_\Phi }), \quad(\lambda-1)\asymp (-k)^{-1}.$$
Consequently, from \eqref{thm3-log-eqn4*} we get
 \begin{multline*}
     \Phi\left(2^{-k\tau}|(u)_{A_k^i}|\right)\leq 2\left(\frac{-k}{-k-1/2}\right)^{p^+_\Phi -1}\Phi\left(2^{-(k+1)\tau}|(u)_{A_{k+1}^j}|\right)\\+C \frac{(-k)^{p^+_\Phi -1}}{2^{kN}}  \int_{A_k^i\cup A_{k+1}^j}\int_{A_k^i\cup A_{k+1}^j}\Phi\left(|D_su(x,y)|x_N^{\alpha_1}y_N^{\alpha_2}\right)d\mu.
 \end{multline*}
 This gives
 \begin{multline*}
 \frac{2^{kN}}{(-k)^{p^+_\Phi -1}}\Phi\left(2^{-k\tau}|(u)_{A_k^i}|\right)\leq \frac{2^{(kN+1)}}{(-k-1/2)^{p^+_\Phi -1}}\Phi\left(2^{ -(k+1)\tau}|(u)_{A_{k+1}^j}|\right)\\+C\int_{A_k^i\cup A_{k+1}^j}\int_{A_k^i\cup A_{k+1}^j}\Phi\left(|D_su(x,y)|x_N^{\alpha_1}y_N^{\alpha_2}\right)d\mu.
 \end{multline*}
Observe that there are $2^{N-1}$ such cubes $A_k^i$ lying below the cube $A_{k+1}^j$. Therefore, summing the above inequalities from $i=2^{N-1}(j-1)+1$ to $2^{N-1}j,$ we obtain 
 \begin{multline*}
\frac{2^{kN}}{(-k)^{p^+_\Phi -1}} \sum^{2^{N-1}j}_{i=2^{N-1}(j-1)+1} \Phi\left(2^{-k\tau}|(u)_{A_k^i}|\right)\leq \frac{2^{N(k+1)}}{(-k-1/2)^{p^+_\Phi -1}}\Phi\left(2^{ -(k+1)\tau}|(u)_{A_{k+1}^j}|\right)\\+C\sum^{2^{N-1}j}_{i=2^{N-1}(j-1)+1}\int_{A_k^i\cup A_{k+1}^j}\int_{A_k^i\cup A_{k+1}^j}\Phi\left(|D_su(x,y)|x_N^{\alpha_1}y_N^{\alpha_2}\right)d\mu.
 \end{multline*}
 Again, summing the above inequalities from $j=1$ to $\sigma_{k+1}$ and using \eqref{eqal-p3}, 
 we obtain 
%  \begin{multline*}
% \frac{2^{kN}}{(-k)^{p^+_\Phi -1}}\sum_{i=1}^{\sigma_{k}}\Phi\left(2^{-k\tau}|(u)_{A_k^i}|\right)\leq \frac{2^{N(k+1)}}{(-k-1/2)^{p^+_\Phi -1}}\sum_{j=1}^{\sigma_{k+1}}\Phi\left(2^{ -(k+1)\tau}|(u)_{A_{k+1}^j}|\right)\\+C\sum_{j=1}^{\sigma_{k+1}}\left(\sum^{2^{N-1}j}_{i=2^{N-1}(j-1)+1}\int_{A_k^i\cup A_{k+1}^j}\int_{A_k^i\cup A_{k+1}^j}\Phi\left(|D_su(x,y)|x_N^{\alpha_1}y_N^{\alpha_2}\right)d\mu\right).
%  \end{multline*}
%  Then, the above inequality gives
\begin{equation}\label{STAR}
    \frac{2^{kN}a_k}{(-k)^{p^+_\Phi -1}}\leq \frac{2^{N(k+1)}a_{k+1}}{(-k-1/2)^{p^+_\Phi -1}}+C\int_{A_k\cup A_{k+1}}\int_{A_k\cup A_{k+1}}\Phi\left(|D_su(x,y)|x_N^{\alpha_1}y_N^{\alpha_2}\right)d\mu,\quad k\in \mathbb{Z}^-.
\end{equation}
Sum up the above inequalities from $k=m$ to $-2$ to obtain
\begin{equation*}
    \sum_{k=m}^{-2} \frac{2^{kN}}{(-k)^{p^+_\Phi -1}}a_k\leq \sum_{k=m}^{-2}\frac{2^{N(k+1)}}{(-k-1/2)^{p^+_\Phi -1}}a_{k+1}+C \int_{\Omega_{n,T}}\int_{\Omega_{n,T}}\Phi(|D_su(x,y)|x_N^{\alpha_1}y_N^{\alpha_2})d\mu.
\end{equation*}
By changing sides and re-indexing, we get  
\begin{multline*}
   \frac{2^{mN}a_m}{(-m)^{p^+_\Phi -1}}+ \sum_{k=m+1}^{-2} \left\{\frac{1}{(-k)^{p^+_\Phi -1}}-\frac{1}{(-k+1/2)^{p^+_\Phi -1}}\right\}2^{kN}a_{k} \\\le\left(\frac{2}{3}\right)^{p^+_\Phi-1}2^{-N}a_{-1}+C\int_{\Omega_{n,T}}\int_{\Omega_{n,T}}\Phi(|D_su(x,y)|x_N^{\alpha_1}y_N^{\alpha_2})d\mu.
\end{multline*}
Now using $$\frac{1}{(-k)^{p^+_\Phi -1}}-\frac{1}{(-k+1/2)^{p^+_\Phi -1}}\asymp \frac{1}{(-k)^{p^+_\Phi }},\qquad \frac{1}{(-m)^{p_\Phi^+}}\le \frac{1}{(-m)^{p_\Phi^+-1}},$$ and adding $2^{-N}a_{-1}$ on both sides, we get 
\begin{equation}\label{eqn-p3}
   \sum_{k=m}^{-1} \frac{2^{kN}}{(-k)^{p^+_\Phi }}a_{k}\le Ca_{-1}+C \int_{\Omega_{n,T}}\int_{\Omega_{n,T}}\Phi(|D_su(x,y)|x_N^{\alpha_1}y_N^{\alpha_2})d\mu.
\end{equation}
Further, use Jensen's inequality and \eqref{H3-p3} to get
\begin{align*}
    a_{-1}= \sum_{i=1}^{\sigma_{-1}}\Phi\left(2^{\tau}|(u)_{A_{-1}^i}|\right)\le C\sum_{i=1}^{\sigma_{-1}}\int_{{A_{-1}^i}}\Phi(|u(x)|)dx\le C\int_{\Omega_{n,T}}\Phi(|u(x)|)dx.
\end{align*}
% Thus, by substituting the value of $a_k$ into \eqref{eqal-p3}, we obtain
% \begin{equation*}
%    \sum_{k=m}^{-1} \frac{2^{kN}}{(-k)^{p^+_\Phi }}\sum_{i=1}^{\sigma_{k}}\Phi\left(2^{-k\tau}|(u)_{A_k^i}|\right)\le C\int_{\Omega_{n,T}}\Phi(|u(x)|)dx+C \int_{\Omega_{n,T}}\int_{\Omega_{n,T}}\Phi(|D_su(x,y)|x_N^{\alpha_1}y_N^{\alpha_2})d\mu.
% \end{equation*}
Hence, the result follows from \eqref{thm3-eq1-p3} by taking $m\to -\infty$.
\end{proof}
Next, we prove a Hardy-type inequality within the domain $B_R(0)^c$. This lemma will be proved by an analogous modification of the proof of Lemma \ref{lem2-p3}. It is a key lemma for establishing $(ii)$ of Theorem \ref{Orlicz-Hardy2*-p3}.
\begin{lem}\label{ext}
Let $N\ge 1,\,s\in (0,1),\,R>0,$ and $\Phi\in \D$ be a Young function. If  $sp^-_\Phi =N,$ then there exists  $C=C(s,N,R,\Phi)>0$ such that for every $u\in \mathcal{C}_c^1(B_R(0)^c),$
\begin{equation*}
        \int_{B_R(0)^c}\Phi\left(\frac {|u(x)|}{|x|^s}\right) \frac{dx}{\ln^{p_\Phi^+}\left(2|x|/R\right)}\le C\int_{B_R(0)^c}\int_{B_R(0)^c}\Phi(|D_su(x,y)|)d\mu+C \int_{B_R(0)^c}\Phi(|u(x)|)dx.
\end{equation*}
\end{lem}
    \begin{proof}
          For each $k\in \mathbb{Z}$, set $$A_k=\{x\in \mathbb{R}^N:2^kR\le |x|<2^{k+1}R\}.$$ Then, we have $$B_R(0)^c=\bigcup_{k=0}^{\infty}A_k.$$ 
          For $x\in A_k$, we have $|x|\ge 2^{k}R,$ which implies 
$$|x|^s\ge 2^{ks}R^{s},\;\;\;\; \ln^{p^+_\Phi}(2|x|/R)\ge (k+1)^{p_\Phi^+}(\ln  2)^{p_\Phi^+}.$$
Therefore, using \eqref{H3-p3}, we get 
\begin{equation}\label{eqn3-3-p3}
   \int_{A_k}\Phi\left(\frac{|u(x)|}{|x|^{s}}\right)\frac{dx}{\ln^{p^+_\Phi}  \left(2|x|/R\right)}\le\frac{\max\{R^{-sp_\Phi^-},R^{-sp_\Phi^+}\}}{(\ln  2)^{p_\Phi^+}(k+1)^{p^+_\Phi }}\int_{A_k}\Phi\left(2^{-ks}|u(x)|\right)dx.
\end{equation}
Next, we estimate the right-hand side of \eqref{eqn3-3-p3}.
Using the triangular inequality and \eqref{H2-p3}, we get 
\begin{align*}
  \Phi(|u(x)|)\le \Phi(|(u)_{A_k}|+|u(x)-(u)_{A_k}|)\leq  2^{p^+_\Phi }\Phi(|(u)_{A_k}|)+2^{p^+_\Phi } \Phi(|u(x)-(u)_{A_k}|).
\end{align*}
Applying Proposition \ref{poincare} with $\Omega=\{x\in \mathbb{R}^N:R<|x|<2R\}$ and $\lambda=2^k$, we obtain
\begin{equation*}
    \int_{A_k}\Phi(|u(x)-(u)_{A_k}|)dx\leq  C \int_{A_k}\int_{A_k}\Phi\left(2^{ks}|D_su(x,y)|\right)d\mu,
\end{equation*}
where $C=C(s,N,R,\Phi)>0$ is a constant. 
By combining the above two inequalities, we obtain
\begin{equation*}
  \int_{A_k}\Phi(|u(x)|)dx\leq 2^{p^+_\Phi }|A_k|\Phi\left(|(u)_{A_k}|\right)+ 2^{p^+_\Phi } C\int_{A_k}\int_{A_k}\Phi\left(2^{ks}|D_su(x,y)|\right)d\mu.
\end{equation*}
Now replace $u$ by $2^{-ks}u$ in the above inequality to obtain
\begin{equation*}
  \int_{A_k}\Phi\left(2^{-ks}|u(x)|\right)dx\leq 2^{p^+_\Phi }|A_k|\Phi\left(2^{-ks}|(u)_{A_k}|\right)+ 2^{p^+_\Phi } C\int_{A_k}\int_{A_k}\Phi\left(|D_su(x,y)|\right)d\mu.
\end{equation*}
Thus, combining \eqref{eqn3-3-p3} with the above inequality and using  $(k+1)^{p_\Phi^+}>1$, we obtain
\begin{equation*}
   \int_{A_k}\Phi\left(\frac{|u(x)|}{|x|^{s}}\right)\frac{dx}{\ln^{p^+_\Phi}  \left(2|x|/R\right)} \leq \frac{C2^{kN}}{(k+1)^{p^+_\Phi }}\Phi\left(2^{-ks}|(u)_{A_k}|\right) + C\int_{A_k}\int_{A_k}\Phi\left(|D_su(x,y)|\right)d\mu.
\end{equation*} 
% For $x\in A_k$ and $k\geq 0,$ we have $$\ln (2|x|/R)\ge (k+1)\ln  2\ge \ln  2.$$  
Let $m\in\mathbb{Z}$ such that $m>1.$ Summing the above inequalities from $k=0$ to $m$, we get
\begin{multline}
    \sum_{k=0}^{m}\int_{A_k}\Phi\left(\frac{|u(x)|}{|x|^{s}}\right)\frac{dx}{\ln^{p^+_\Phi}  \left(2|x|/R\right)} \leq C\sum_{k=0}^{m}\frac{2^{kN}}{(k+1)^{p^+_\Phi }}\Phi\left(2^{-ks}|(u)_{A_k}|\right)\\ + C\int_{B_R(0)^c}\int_{B_R(0)^c}\Phi\left(|D_su(x,y)|\right)d\mu.\label{log-eq3-p3}
\end{multline} 
We next estimate the first term on the right-hand side of \eqref{log-eq3-p3}.  
By triangular inequality and Lemma \ref{Lem-H-p3} with $\Lambda=2^{p^+_\Phi }$,  for every $\lambda\in (1,2^{p^+_\Phi })$, we get 
\begin{multline*}
    \Phi\left(2^{-(k+1)s}|(u)_{A_{k+1}}|\right)\leq  \Phi\left(2^{-(k+1)s}|(u)_{A_{k}}|+2^{-(k+1)s}\left|(u)_{A_{k+1}}- (u)_{A_k}\right|\right)\\\leq \lambda\Phi\left(2^{-(k+1)s}|(u)_{A_{k}}|\right)+\frac{C}{\left(\lambda-1\right)^{p^+_\Phi -1}}\Phi\left(2^{-(k+1)s}\left|(u)_{A_{k+1}}-(u)_{A_k}\right|\right).
\end{multline*}
 Now use  \eqref{H3-p3} and $sp^-_\Phi=N$ to get
\begin{align*}
    \Phi\left(2^{-(k+1)s}|(u)_{A_{k}}|\right)\le 2^{-sp^-_\Phi }\Phi\left( 2^{ -ks}|(u)_{A_{k}}|\right)=  2^{-N}\Phi\left(2^{ -ks}|(u)_{A_{k}}|\right).
\end{align*} 
Further, applying Lemma \ref{Disjointset} with $P= A_k$ and $Q = A_{k+1}$, and using Proposition \ref{poincare} with $\Omega=\{x\in \mathbb{R}^N:R<|x|<4R\}$ and $\lambda=2^k,$ we get a constant $C>0$ (independent of $k$) so that
\begin{align*}
\Phi\left(\left|(u)_{A_{k+1}}- (u)_{A_k}\right|\right)&\leq \frac{C}{2^{kN}}\int_{A_k\cup A_{k+1}}\int_{A_k\cup A_{k+1}}\Phi\left(2^{ks}|D_su(x,y)|\right)d\mu.
  \end{align*}  
 By combining the above three inequalities, for every $\lambda\in (1,2^{p^+_\Phi })$, we obtain
\begin{multline*}
\Phi\left(2^{-(k+1)s}|(u)_{A_{k+1}}|\right)\leq \lambda 2^{-N}\Phi\left(2^{-ks}|(u)_{A_{k}}|\right) \\+ \frac{C}{2^{kN}\left(\lambda-1\right)^{p^+_\Phi -1}}\int_{A_k\cup A_{k+1}}\int_{A_k\cup A_{k+1}}\Phi\left(|D_su(x,y)|\right)d\mu.
\end{multline*}
Now, for each $ k\ge 0$, we choose
 $$\lambda(k):=\left(\frac{k+2}{k+3/2}\right)^{p^+_\Phi -1}.$$ Then, one can verify that   $\lambda\in (1,2^{p^+_\Phi })$ and $(\lambda-1)\asymp (k+2)^{-1}.$
Consequently, 
\begin{multline*}
\Phi\left(2^{ -(k+1)s}|(u)_{A_{k+1}}|\right)\leq 2^{-N}\left(\frac{k+2}{k+3/2}\right)^{p^+_\Phi -1}\Phi\left(2^{-ks}|(u)_{A_{k}}|\right)\\+C\frac{(k+2)^{p^+_\Phi -1}}{2^{kN}}  \int_{A_k\cup A_{k+1}}\int_{A_k\cup A_{k+1}}\Phi\left(|D_su(x,y)|\right)d\mu.
\end{multline*}
% This yields
% \begin{multline*}
% \frac{2^{(k+1)N}}{(k+2)^{p^+_\Phi -1}}\Phi\left(2^{-(k+1)s}|(u)_{A_{k+1}}|\right)\leq \frac{2^{kN}}{(k+3/2)^{p^+_\Phi -1}}\Phi\left(2^{-ks}|(u)_{A_{k}}|\right)\\+C\int_{A_k\cup A_{k+1}}\int_{A_k\cup A_{k+1}}\Phi\left(|D_su(x,y)|\right)d\mu.
% \end{multline*}
Multiplying the above inequality by $2^{(k+1)N}(k+2)^{1-p^+_\Phi}$ and then take the sum to obtain
\begin{multline*}
\sum_{k=0}^{m}\frac{2^{(k+1)N}}{(k+2)^{p^+_\Phi -1}}\Phi\left(2^{-(k+1)s}|(u)_{A_{k+1}}|\right)\leq \sum_{k=0}^{m} \frac{2^{kN}}{(k+3/2)^{p^+_\Phi -1}}\Phi\left(2^{-ks}|(u)_{A_{k}}|\right)\\+C2^N\int_{B_R(0)^c}\int_{B_R(0)^c}\Phi\left(|D_su(x,y)|\right)d\mu.
\end{multline*}
By changing sides and re-indexing, we get
\begin{multline*}
   \frac{2^{(m+1)N}}{(m+2)^{p^+_\Phi -1}}\Phi\left(2^{-(m+1)s}|(u)_{A_{m+1}}|\right)+ \sum_{k=1}^{m}\left\{\frac{2^{kN}}{(k+1)^{p^+_\Phi -1}}-\frac{2^{kN}}{(k+3/2)^{p^+_\Phi -1}}\right\}\Phi\left(2^{-ks}|(u)_{A_k}|\right)\\\leq\left(\frac{2}{3}\right)^{p^+_\Phi-1} \Phi(|(u)_{A_0}|)+ C\int_{B_R(0)^c}\int_{B_R(0)^c}\Phi\left(|D_su(x,y)|\right)d\mu.
\end{multline*} 
Clearly, the first term on the left-hand side of the above inequality is positive.
 Thus, adding $\Phi(|(u)_{A_0}|)$ to both sides of the above inequality and using 
\begin{equation*}
    \frac{1}{(k+1)^{p^+_\Phi -1}}-\frac{1}{(k+3/2)^{p^+_\Phi -1}}\asymp  \frac{1}{(k+1)^{p^+_\Phi }},\quad\forall\,k\in\mathbb{Z}^+,
\end{equation*}
 we obtain
$$\sum_{k=0}^{m}\frac{2^{kN}}{(k+1)^{p^+_\Phi }}\Phi\left(2^{-ks}|(u)_{A_k}|\right)\le C\Phi(|(u)_{A_0}|)+C\int_{B_R(0)^c}\int_{B_R(0)^c}\Phi\left(|D_su(x,y)|\right)d\mu.$$
Further, by Jensen's inequality  $$\Phi(|(u)_{A_0}|)\leq \frac{1}{|A_0|}\int_{A_0}\Phi(|u(x)|)dx\le C\int_{B_R(0)^c}\Phi(|u(x)|)dx,$$
for some $C=(N,R).$ 
Hence, the result follows from \eqref{log-eq3-p3} by taking $m\to \infty$.
    \end{proof}

\section{Fractional Orlicz boundary Hardy-type inequalities}\label{orlicz-boundary-p3}
This section will prove all the theorems stated in the introduction. Before proving Theorem \ref{Orlicz-Hardy2-p3} and Theorem \ref{Orlicz-Hardy1*-p3}, we need to recall the definition of a Lipschitz domain $\Omega$ (see \cite[Page 581]{Dyda2007}). 

\smallskip

Assume that $\Omega$ is a bounded Lipschitz domain. Then for each $x\in \partial \Omega$ there exists $r^\prime_x>0,$ an isometry $T_x$ of $\mathbb{R}^N$, and a Lipschitz function $\gamma_x:\mathbb{R}^{N-1}\rightarrow \mathbb{R}$ so that 
$$T_x(\Omega)\cap B_{r_x^\prime}(T_x(x))=\{\xi:\xi_N>\gamma_x(\xi^\prime)\}\cap B_{r_x^\prime}(T_x(x)).$$ 
 Without
loss of generality, we can assume $T_x$ to be the identity map. Otherwise, we can work with the Lipschitz
domain $T_x(\Omega)$ and the point $T_x(x)$, then pull back the construction to the original domain $\Omega$ and the point
$x$ via $T^{-1}_x$. %(see \cite[Page 10]{roy2022}).
Then 
\begin{equation}\label{eqn1-s4p3}
\Omega\cap B_{r_x^\prime}(x)=\{\xi:\xi_N>\gamma_x(\xi^\prime)\}\cap B_{r_x^\prime}(x)    
\end{equation}
and 
$\partial\Omega\subset\cup_{x\in\partial\Omega}B_{r^\prime_x}(x)$.  Choose $0<r_x<1$ such that $r_x\le r^\prime_x$ and for all $y\in \Omega\cap B_{r_x}(x),$ there exists $z\in \partial\Omega\cap B_{r_x}(x)$ satisfying $\delta_\Omega(y)=|y-z|.$ Then $\partial\Omega\subset\cup_{x\in\partial\Omega}B_{r_x}(x).$ As $\partial\Omega$ is compact, there exist $x_1,\ldots,x_n\in \partial\Omega$ such that $$\partial\Omega\subset\bigcup_{i=1}^nB_{r_i}(x_i),\quad \text{where}\; r_i=r_{x_i}.$$

\smallskip

\noindent\textbf{Proof of Theorem \ref{Orlicz-Hardy2-p3}.}  Let $\Omega$ be a bounded Lipschitz domain, $sp^+_\Phi=1$, and $R\ge R_\Omega:=\sup\{\delta_\Omega(x):x\in\Omega\}$. Let  $u\in \mathcal{C}^1_c(\Omega)$ and   $\overline{\Omega_0}\subset \Omega$ such that $\Omega\subset \cup_{i=0}^n\Omega_i$, where $\Omega_i=B_{r_i}(x_i)$, $i\in\{1,\ldots,n\}.$  Let $\{\eta_i\}_{i=0}^n$ be the associate partition of unity. Then, we have
 \begin{equation}\label{parti}
u=\sum_{i=0}^nu_i,\;\;\text{where}\;\;u_i=\eta_iu     
 \end{equation}
and hence $supp(u_i)\subset \Omega\cap \Omega_i$.
Now using \eqref{H2-p3}, we obtain 
$$\int_{\Omega}\Phi\left(\frac{|u(x)|}{\delta^s_\Omega(x)}\right)\frac{dx}{\ln^{p^+_\Phi }(2R/\delta_\Omega(x))}\le C\sum_{i=0}^n \int_{\Omega}\Phi\left(\frac{|u_i(x)|}{\delta^s_\Omega(x)}\right)\frac{dx}{\ln^{p^+_\Phi }(2R/\delta_\Omega(x))}.$$ 
Further, applying Lemma \ref{lem-ext}, we obtain 
\begin{multline}\label{eqn1-thm1-3}
   \int_\Omega\Phi(|u_i(x)|)dx+ \int_{\Omega}\int_{\Omega}\Phi\left(|D_s(u_i)(x,y)|\right)d\mu\leq C\int_{\Omega}\int_{\Omega}\Phi\left(|D_su(x,y)|\right)d\mu\\+C\int_\Omega\Phi(|u(x)|)dx,\quad 0\le  i \le n.
\end{multline}
Thus, it is sufficient to prove the required inequality \eqref{2} for all $u_i$. Since $\overline{\Omega_0}\subset \Omega$, there exist positive constants $C_1$ and $C_2$ such that
\begin{equation}\label{eqiv1-p3}
   C_1\le \delta_\Omega(x)\le C_2,\quad \forall\,x\in \Omega_0. 
\end{equation}
Therefore,
  $$\int_{\Omega_0}\Phi\left(\frac{|u_0(x)|}{\delta^s_\Omega(x)}\right)\frac{dx}{\ln^{p^+_\Phi}(2R/\delta_\Omega(x))}\asymp \int_{\Omega_0}\Phi\left(|u_0(x)|\right) dx.$$
To prove \eqref{2} for other $u_i$, consider the transformation  $F:\mathbb{R}^N\rightarrow\mathbb{R}^N$ such that $F(x^\prime,x_N)=(x^\prime,x_N-\gamma_{x_i}(x^\prime))$. 
Then, by Lemma \ref{lips}, we have 
$$\delta_\Omega(x)\asymp    x_N-\gamma_{x_i}(x^\prime)=\xi_N,\quad \forall\,x\in \Omega\cap \Omega_i,$$ where $F(x)=(\xi_1,\ldots,\xi_N)$. 
Observe that, $F$ is an invertible Lipschitz function and $F(\Omega\cap\Omega_i)\subset \mathbb{R}^{N-1}\times (0,\infty)$ (see \eqref{eqn1-s4p3}).
 Thus, applying  Lemma \ref{changeofvar} and Lemma \ref{lem2-p3} with $\alpha_1=\alpha_2=0$, we get
 \begin{multline*}
     \int_{\Omega\cap \Omega_i}\Phi\left(\frac{|u_i(x)|}{\delta^s_\Omega(x)}\right)\frac{dx}{\ln^{p^+_\Phi }(2R/\delta_\Omega(x))} 
 \leq C\int_{F(\Omega\cap\Omega_i)}\Phi\left(\frac{|u_i\circ F^{-1}(\xi)|}{\xi_N^{s}}\right)\frac{d\xi}{\ln^{p^+_\Phi }(2R/\xi_N)}\\\leq  C\int_{F(\Omega\cap\Omega_i)} \Phi(|u_i\circ F^{-1}(\xi)|)d\xi+C\int_{F(\Omega\cap\Omega_i)}\int_{F(\Omega\cap\Omega_i)}\Phi(|D_s(u_i\circ F^{-1})(x,y)|)d\mu\\=C\int_{\Omega\cap\Omega_i} \Phi(|u_i(x)|)dx+C\int_{\Omega\cap\Omega_i}\int_{\Omega\cap\Omega_i}\Phi(|D_su_i(x,y)|)d\mu.
 \end{multline*}
This completes the proof.
\qed

\smallskip

\noindent\textbf{Proof of Theorem \ref{Orlicz-Hardy1*-p3}.}  Let $\Omega$ be a bounded Lipschitz domain and $sp^+_\Phi<1.$ Let  $u\in \mathcal{C}^1_c(\Omega).$ In this proof, we use the decomposition of $u$ defined in the proof of Theorem \ref{Orlicz-Hardy2-p3}.  By using \eqref{parti} and \eqref{H2-p3}, we get
$$\int_{\Omega}\Phi\left(\frac{|u(x)|}{\delta^s_\Omega(x)}\right)dx\le C \sum_{i=0}^n\int_{\Omega}\Phi\left(\frac{|u_i(x)|}{\delta^s_\Omega(x)}\right)dx.$$ 
Since $u_i$ satisfies \eqref{eqn1-thm1-3}, it is sufficient to prove the required inequality \eqref{1} for all $u_i$.
 Use \eqref{eqiv1-p3} to get
 $$\int_{\Omega_0}\Phi\left(\frac{|u_0(x)|}{\delta^s_\Omega(x)}\right)dx\asymp \int_{\Omega_0}\Phi\left(|u_0(x)|\right) dx,$$
and hence \eqref{1} is proved for $u_0.$
 To prove \eqref{1} for other $u_i$, consider the transformation $F:\mathbb{R}^N\rightarrow\mathbb{R}^N$ such that $F(x^\prime,x_N)=(x^\prime,x_N-\gamma_{x_i}(x^\prime))$.
Then, by Lemma \ref{lips}, we have 
 $$\delta_\Omega(x)\asymp    x_N-\gamma_{x_i}(x^\prime)=\xi_N,\quad \forall\,x\in \Omega\cap \Omega_i,$$ where $F(x)=(\xi_1,\ldots,\xi_N)$. 
Notice that $F$ is an invertible Lipschitz function and $F(\Omega\cap\Omega_i)\subset \mathbb{R}^{N-1}\times (0,\infty)$ (see \eqref{eqn1-s4p3}). Therefore, applying Lemma \ref{changeofvar} and Lemma \ref{lem3}, we obtain  
 \begin{multline*}
     \int_{\Omega\cap \Omega_i}\Phi\left(\frac{|u_i(x)|}{\delta^s_\Omega(x)}\right)dx
 \asymp \int_{F(\Omega\cap\Omega_i)}\Phi\left(\frac{|u_i\circ F^{-1}(\xi)|}{\xi_N^{s}}\right)d\xi\\\leq  C\int_{F(\Omega\cap\Omega_i)} \Phi(|u_i\circ F^{-1}(\xi)|)d\xi+\int_{F(\Omega\cap\Omega_i)}\int_{F(\Omega\cap\Omega_i)}\Phi(|D_s(u_i\circ F^{-1})(x,y)|)d\mu\\=C\int_{\Omega\cap\Omega_i} \Phi(|u_i(x)|)dx+C\int_{\Omega\cap\Omega_i}\int_{\Omega\cap\Omega_i}\Phi(|D_su_i(x,y)|)d\mu.
 \end{multline*}
%Consequently, \eqref{1} holds for all $u_i$. 
Hence, the proof.
\qed

\smallskip

 We prove the following proposition before giving the proof for Theorem \ref{Orlicz-Hardy2*-p3}.
 \begin{prop}\label{Thm1.2-p3}
Let $N\ge 1,\,s\in (0,1),$ and $\Omega$ be a bounded Lipschitz domain. For any Young function $\Phi\in \D,$ if $sp_\Phi^->1$
then the fractional boundary Hardy inequality \eqref{Hardy-p3} holds.
 \end{prop}
 \begin{proof}
    It follows from Theorem 1.2 of \cite{roy2022} that if $\Phi$ satisfies
     \begin{equation}\label{eqn-thm-p3}
          \lim \inf_{\lambda\to 0+}\left\{\sup_{t>0}\frac{\Phi(\lambda t)}{\lambda^\frac{1}{s}\Phi(t)}\right\}=0,
      \end{equation}
  then \eqref{Hardy-p3} holds. By using \eqref{H3-p3}, we get 
  \begin{equation*}
     \frac{\Phi(\lambda t)}{\lambda^s\Phi(t)}\le \frac{\lambda^{p_\Phi^-}\Phi(t)}{\lambda^\frac{1}{s}\Phi(t)}=\lambda^\frac{sp^-_\Phi-1}{s},\quad\forall\,\lambda\in (0,1),\;\forall\,t>0.
  \end{equation*}
 If $sp^-_\Phi>1$, then the above inequality implies that $\Phi$ satisfies \eqref{eqn-thm-p3}. This completes the proof.
 \end{proof}
\smallskip

\noindent\textbf{Proof of Theorem \ref{Orlicz-Hardy2*-p3}.} $(i)$ Let $\Omega$ be the domain above the graph of a Lipschitz function $\gamma:\mathbb{R}^{N-1}\rightarrow\mathbb{R},$ $sp^+_\Phi=1,$ and $R>0$. 
Let $u\in \mathcal{C}_c^1(\Omega)$ such that $supp(u)\subset \Omega_R=\{x\in \Omega:\delta_\Omega(x)<R\}.$ Choose  $R_1,R_2\in (0,1)$ such that $R_1<R_2<R.$  Let $\chi\in \mathcal{C}_c^\infty(\Omega)$ such that $\chi(x)=1$ for all $x\in \overline{\Omega_{R_1}},$  $\chi(x)=0$ for all $x\in \Omega\setminus\Omega_{R_2},$ and $0\le \chi\leq 1.$ Then 
$$u=u_1+u_2,\quad\text{where}\;u_1=\chi u,\,u_2=(1-\chi)u.$$ Thus, $supp(u_1)\subset \Omega_{R_2}$ and $supp(u_2)\subset \Omega_R\setminus\Omega_{R_1}.$ Now, use \eqref{H2-p3} to get
\begin{equation*}
    \int_{\Omega}\Phi\left(\frac{|u(x)|}{\delta^s_\Omega(x)}\right)\frac{dx}{\ln^{p^+_\Phi }(\rho(x))}\le C\sum_{i=1}^2 \int_{\Omega}\Phi\left(\frac{|u_i(x)|}{\delta^s_\Omega(x)}\right)\frac{dx}{\ln^{p^+_\Phi }(\rho(x))},
\end{equation*}
where $\rho(x)=\frac{2R}{\delta_\Omega(x)}.$ Therefore, according to Lemma \ref{lem-ext}, it suffices to prove \eqref{3} for $u_1$ and $u_2$.
To prove \eqref{3} for $u_1$, consider the transformation $F:\mathbb{R}^N\rightarrow\mathbb{R}^N$ such that $F(x^\prime,x_N)=(x^\prime,x_N-\gamma(x^\prime))$. Then, by  Lemma \ref{lips}, 
 $$\delta_\Omega(x)\asymp\xi_N,\quad \forall\,x\in  \Omega,$$ where $F(x)=(\xi_1,\ldots,\xi_N).$
Thus, $F(\Omega_{R_2})\subset\mathbb{R}^{N-1}\times(0,1)$, and there exists $n_1\in \mathbb{N}$ such that $supp(u_1\circ F^{-1})\subset (-n_1,n_1)^{N-1}\times (0,1)$ on $F(\Omega_{R_2}).$   
Now, applying  Lemma \ref{changeofvar} and Lemma \ref{lem2-p3}  with $\alpha_1=\alpha_2=0$, we obtain 
 \begin{multline*} \int_{\Omega_1}\Phi\left(\frac{|u_1(x)|}{\delta^s_\Omega(x)}\right)\frac{dx}{\ln^{p^+_\Phi }(2R/\delta_\Omega(x))} 
 \leq  C\int_{F(\Omega_1)}\Phi\left(\frac{|u_1\circ F^{-1}(\xi)|}{\xi_N^{s}}\right)\frac{d\xi}{\ln^{p^+_\Phi }(2R/\xi_N)}\\\leq C\int_{F(\Omega_1)} \Phi(|u_1\circ F^{-1}(\xi)|)d\xi+C\int_{F(\Omega_1)}\int_{F(\Omega_1)}\Phi(|D_s(u_1\circ F^{-1})|)d\mu\\=C\int_{\Omega_1} \Phi(|u_1(x)|)dx+C\int_{\Omega_1}\int_{\Omega_1}\Phi(|D_su_1(x,y)|)d\mu.
 \end{multline*}
 Since    $R_1\le \delta_\Omega(x)< R,\;\forall\, x\in \Omega_R\setminus\Omega_{R_1},$
we have $$\int_{\Omega_R\setminus\Omega_{R_1}}\Phi\left(\frac{|u_2(x)|}{\delta^s_\Omega(x)}\right)\frac{dx}{\ln^{p^+_\Phi }(2R/\delta_\Omega(x))} 
 \asymp \int_{\Omega_R\setminus\Omega_{R_1}}\Phi\left(|u_2(x)|\right) dx.$$
Hence, the conclusion follows.
\smallskip

\noindent$(ii)$ Let  
$\Omega\in \mathcal{A}_3,$ $sp^-_\Phi=N>1,$ and  $R>0$. Choose $r>1$  such that $\Omega^c\subset B_{rR}(0)$. 
Let  $u\in \mathcal{C}_c^1(\Omega)$ and $\chi\in \mathcal{C}_c^\infty(B_{2rR}(0))$ be such that $\chi(x)=1$ for all $x\in B_{rR}(0)$ and $0\le \chi\le 1.$ 
Then
$$u=u_1+u_2,\quad\text{where}\;u_1=\chi u,\,u_2=(1-\chi)u.$$ 
So, $supp(u_1)\subset  B_{2rR}(0)\cap\Omega$ and $supp(u_2)\subset B_{rR}(0)^c.$   
Thus, due to Lemma \ref{lem-ext}, it is sufficient to prove \eqref{4} for $u_1$ and $u_2$. Set $\rho(x)=\max\left\{\frac{2R}{\delta_\Omega(x)},\frac{2\delta_\Omega(x)}{R}\right\}$. Then, we have $\rho(x)\ge 2$, and $\delta_{B_{2rR}(0)\cap\Omega}(x)\le \delta_\Omega(x)$ for every $x\in B_{2rR}(0)\cap\Omega.$ Since $B_{2rR}(0)\cap\Omega$ is a bounded Lipschitz domain and $sp^-_\Phi>1$, by Proposition \ref{Thm1.2-p3}, we get
\begin{multline*}
    \int_{B_{2rR}(0)\cap\Omega}\Phi\left(\frac{|u_1(x)|}{\delta^s_\Omega(x)}\right)\frac{dx}{\ln^{p^+_\Phi}(\rho(x))}\le C\int_{B_{2rR}(0)\cap\Omega}\Phi\left(\frac{|u_1(x)|}{\delta^s_{B_{2rR}(0)\cap\Omega}(x)}\right)dx
    \\\le C\int_{B_{2rR}(0)\cap\Omega}\int_{B_{2rR}(0)\cap\Omega}\Phi(|D_su_1(x,y)|)d\mu.
\end{multline*}
To prove \eqref{4} for $u_2,$ observe that $\delta_\Omega(x)\geq C|x|,\;\forall\,x\in B_{rR}(0)^c$. Moreover, one can verify that
$$\ln^{p^+_\Phi}\left(\rho(x)\right)> C\ln^{p^+_\Phi}\left(\frac{2|x|}{rR}\right),\quad\forall\,x\in B_{rR}(0)^c.$$ Thus, applying
 Lemma \ref{ext}, we get
\begin{multline*}
    \int_{B_{rR}(0)^c}\Phi\left(\frac{|u_2(x)|}{\delta^s_\Omega(x)}\right)\frac{dx}{\ln^{p^+_\Phi}(\rho(x))}\le C\int_{B_{rR}(0)^c}\Phi\left(\frac{|u_2(x)|}{|x|^s}\right)\frac{dx}{\ln^{p^+_\Phi}(2|x|/rR)}
    \\\le C\int_{B_{rR}(0)^c}\int_{B_{rR}(0)^c}\Phi(|D_su_2(x,y)|)d\mu+C\int_{B_{rR}(0)^c}\Phi(|u_2(x)|)dx.
\end{multline*}
Consequently, the inequality \eqref{4} also holds for $u_2$. This completes the proof.
\qed

\smallskip

\noindent\textbf{Proof of Theorem \ref{Orlicz-Hardy3-p3}.} 
Let $u \in \mathcal{C}_c^1(\mathbb{R}^N_+)$ such that $\text{supp}(u)\subset \mathbb{R}^{N-1}\times (0,R)$. Let $n_0>0$ such that $\frac{2n_0}{ R}\in \mathbb{N}$ and $\text{supp}(u)\subset (-n_0,n_0)^{N-1}\times (0,R)$. From the definitions of $A_k$ and $A_k^i$ as defined in Section \ref{flatbunndary-p3} with $n=n_0$ and $T=R>0$, we have
$$(-n_0,n_0)^{N-1}\times (0,R)=\bigcup_{k=-\infty}^{-1}A_k,\qquad A_k=\bigcup_{i=1}^{\sigma_k}A_k^i,$$ where $\sigma_k=2^{(-k+1)(N-1)}n_0^{N-1}R^{1-N}$. Let $\tau=s-\alpha_1-\alpha_2$ and $m\in \mathbb{Z}$ such that $m<-2.$ 
Since $\tau p^+_\Phi=1,$ from \eqref{thm3-eq1-p3} we get
\begin{equation}\label{hass}
\sum_{k=m}^{-1}\int_{A_k}\Phi\left(\frac{|u(x)|}{x_N^\tau}\right)\frac{dx}{\ln^{p^+_\Phi}(2R/x_N)}\le C\sum_{k=m}^{-1}\frac{2^{kN}}{(-k)^{p^+_\Phi }}a_k+C\int_{\mathbb{R}^N_+}\int_{\mathbb{R}^N_+} \Phi\left(|D_su(x,y)|x_N^{\alpha_1}y_N^{\alpha_2}\right)d\mu,
\end{equation}
where $a_k=\sum_{i=1}^{\sigma_{k}}\Phi\left(2^{-k\tau}|(u)_{A_k^i}|\right).$
Thus it is sufficient to prove that 
\begin{equation}\label{starhass}
    \sum_{k=m}^{-1}\frac{2^{kN}}{(-k)^{p^+_\Phi }}a_k\leq C\int_{\mathbb{R}^N_+}\int_{\mathbb{R}^N_+} \Phi\left(|D_su(x,y)|x_N^{\alpha_1}y_N^{\alpha_2}\right)d\mu.
\end{equation}
%To prove \eqref{starhass}, let  Then, f
From \eqref{STAR} we have
\begin{equation}
    \frac{2^{kN}a_k}{(-k)^{p^+_\Phi -1}}\leq \frac{2^{N(k+1)}a_{k+1}}{(-k-1/2)^{p^+_\Phi -1}}+C\int_{A_k\cup A_{k+1}}\int_{A_k\cup A_{k+1}}\Phi\left(|D_su(x,y)|x_N^{\alpha_1}y_N^{\alpha_2}\right)d\mu,\quad k\in \mathbb{Z}^-.
\end{equation}
Sum up the above inequalities from $k=m$ to $-1$ to obtain
\begin{equation*}
    \sum_{k=m}^{-1} \frac{2^{kN}a_k}{(-k)^{p^+_\Phi -1}}\leq \sum_{k=m}^{-1}\frac{2^{N(k+1)}a_{k+1}}{(-k-1/2)^{p^+_\Phi -1}}+C\int_{\mathbb{R}^N_+}\int_{\mathbb{R}^N_+}\Phi\left(|D_su(x,y)|x_N^{\alpha_1}y_N^{\alpha_2}\right)d\mu.
\end{equation*}
By changing sides, rearranging, and re-indexing, we get 
\begin{multline*}
  \frac{2^{mN}a_m}{(-m)^{p^+_\Phi -1}}+ \sum_{k=m+1}^{-1} \left\{\frac{1}{(-k)^{p^+_\Phi -1}}-\frac{1}{(-k+1/2)^{p^+_\Phi -1}}\right\}2^{kN}a_{k}\\\le2^{p^+_\Phi -1}a_0+C\int_{\mathbb{R}^N_+}\int_{\mathbb{R}^N_+}\Phi\left(|D_su(x,y)|x_N^{\alpha_1}y_N^{\alpha_2}\right)d\mu.
\end{multline*}
Now, using $$\frac{1}{(-k)^{p^+_\Phi -1}}-\frac{1}{(-k+1/2)^{p^+_\Phi -1}}\asymp \frac{1}{(-k)^{p^+_\Phi }},\quad\forall\,k\in \mathbb{Z}^-$$ and 
$(-m)^{p_\Phi^+-1}\le (-m)^{p_\Phi^+},$ we get 
\begin{equation*}
     \sum_{k=m}^{-1} \frac{2^{kN}}{(-k)^{p^+_\Phi }}a_{k}\le2^{p^+_\Phi -1}a_0+C\int_{\mathbb{R}^N_+}\int_{\mathbb{R}^N_+}\Phi\left(|D_su(x,y)|x_N^{\alpha_1}y_N^{\alpha_2}\right)d\mu.
\end{equation*}
Since $supp(u)\subset (-n_0,n_0)^{N-1}\times (0,R)$ and $ A^i_0\subset \mathbb{R}^{N-1}\times [R,2R)$, we have $(u)_{A^i_0}=0$ and hence $$a_0=\sum_{i=1}^{\sigma_0}\Phi\left(|(u)_{A_0^i}|\right)=0.$$ Consequently, \eqref{starhass} follows. This completes the proof.
% Therefore, 
% \begin{equation*}
%    \sum_{k=m}^{-1} \frac{2^{kN}}{(-k)^{p^+_\Phi }}a_{k}\le C \int_{\mathbb{R}^N_+}\int_{\mathbb{R}^N_+}\Phi\left(|D_su(x,y)|x_N^{\alpha_1}y_N^{\alpha_2}\right)d\mu.
% \end{equation*}
% Thus, from \eqref{thm3-eq1-p3} we get
% \begin{equation*}
% \sum_{k=m}^{-1}\int_{A_k}\Phi\left(\frac{|u(x)|}{x_N^\tau}\right)\frac{dx}{\ln^{p^+_\Phi }(2R/x_N)}\le C \int_{\mathbb{R}^N_+}\int_{\mathbb{R}^N_+}\Phi\left(|D_su(x,y)|x_N^{\alpha_1}y_N^{\alpha_2}\right)d\mu.
% \end{equation*}
% Hence, the conclusion follows by taking $m\to -\infty$.
\qed

\begin{center}
    \textbf{Acknowledgement}
\end{center}

The author would like to thank Dr. Anoop T.V. (IIT Madras) and Dr. Prosenjit Roy (IIT Kanpur) for the valuable discussions.

\bibliographystyle{abbrvurl}
\bibliography{Reference}

\end{document}